\documentclass[a4paper]{amsart}
\author{Dilip Raghavan}
\thanks{Author partially supported by National University of Singapore 
research grant number R-146-000-211-112.}
\address{Department of Mathematics \\
National University of Singapore\\
Singapore 119076}
\email{raghavan@math.nus.edu.sg}
\urladdr{http://www.math.toronto.edu/raghavan}
\date{\today}
\title{More on the density zero ideal}
\usepackage{amssymb, amsmath, amsthm, mathrsfs, enumerate, amsfonts, latexsym, bbm, rotating}
\DeclareSymbolFont{symbolsC}{U}{txsyc}{m}{n}
\SetSymbolFont{symbolsC}{bold}{U}{txsyc}{bx}{n}
\DeclareFontSubstitution{U}{txsyc}{m}{n}
\DeclareMathSymbol{\multimapdot}{\mathrel}{symbolsC}{20}
\newtheorem{Theorem}{Theorem}
\newtheorem{Claim}[Theorem]{Claim}
\newtheorem{Lemma}[Theorem]{Lemma}
\newtheorem{Cor}[Theorem]{Corollary}

\newtheorem{Question}[Theorem]{Question}

\theoremstyle{definition}
\newtheorem{Def}[Theorem]{Definition}

\theoremstyle{remark}
\newtheorem{remark}[Theorem]{Remark}

\newcommand{\forces}{\Vdash}
\newcommand{\restrict}{\mathord{\upharpoonright}}
\newcommand{\forallbutfin}{{\forall}^{\infty}}
\newcommand{\existsinf}{{\exists}^{\infty}}

\renewcommand{\b}{\mathfrak{b}}
\renewcommand{\d}{{\mathfrak{d}}}
\newcommand{\s}{\mathfrak{s}}
\newcommand{\spr}{\mathfrak{s}\mathord{\left(\mathfrak{pr}\right)}}

\newcommand{\rr}{{\mathfrak{r}}}
\newcommand{\uu}{{\mathfrak{u}}}

\newcommand{\p}{{\mathfrak{p}}}

\renewcommand{\[}{\left[}
\renewcommand{\]}{\right]}
\renewcommand{\P}{\mathbb{P}}

\newcommand{\Q}{\mathbb{Q}}
\newcommand{\lc}{\left|}
\newcommand{\rc}{\right|}

\newcommand\ZFC{\mathrm{ZFC}}

\newcommand\CH{\mathrm{CH}}
   
\newcommand{\BS}{{\omega}^{\omega}}
\DeclareMathOperator{\non}{non}

\DeclareMathOperator{\cov}{cov}
\DeclareMathOperator{\add}{add}

\DeclareMathOperator{\dom}{dom}
\DeclareMathOperator{\ran}{ran}

\newcommand{\Pset}{\mathcal{P}}

\newcommand{\NNN}{\mathcal{N}}

\newcommand{\AAA}{{\mathcal{A}}}
\newcommand{\C}{{\mathscr{C}}}

\newcommand{\GGG}{{\mathcal{G}}}
\newcommand{\ZZZ}{{\mathcal{Z}}}

\newcommand{\cube}{{\[\omega\]}^{\omega}}

\newcommand{\I}{{\mathcal{I}}}

\newcommand{\F}{{\mathcal{F}}}

\newcommand{\V}{{\mathbf{V}}}

\newcommand{\VG}{{{\mathbf{V}}[G]}}
\newcommand{\VP}{{\mathbf{V}}^{\P}}

\newcommand{\Ss}{\mathbb{S}}
\newcommand{\RR}{\mathbb{R}}

\newcommand{\pr}[2]{\langle #1, #2 \rangle}
\newcommand{\seq}[4]{\langle {#1}_{#2}: #2 #3 #4 \rangle}
\begin{document}
\begin{abstract}
The main result of this paper is an improvement of the upper bound on the cardinal invariant ${\cov}^{\ast}({\ZZZ}_{0})$ that was discovered in~\cite{cov_z0}.
Here ${\ZZZ}_{0}$ is the ideal of subsets of the set of natural numbers that have asymptotic density zero.
This improved upper bound is also dualized to get a better lower bound on the cardinal ${\non}^{\ast}({\ZZZ}_{0})$.
En route some variations on the splitting number are introduced and several relationships between these variants are proved.
\end{abstract}
\maketitle
\section{Introduction} \label{sec:intro}
We use $\omega$ to denote the set of natural number in keeping with usual set-theoretic convention.
Recall that a set $A \subset \omega$ is said to have \emph{asymptotic density 0} if $\displaystyle\lim_{n \rightarrow \infty}{\displaystyle\frac{\lc A \cap n \rc}{n}} = 0$.
By ${\ZZZ}_{0}$ we denote the set $\left\{ A \subset \omega: A \ \text{has asymptotic density} \ 0 \right\}$.
Recall that given a set $a$, $\I$ is said to be an \emph{ideal} on $a$ if $\I$ is a subset of $\Pset(a)$ such that the following conditions hold: if $b \subseteq a$ is finite, then $b \in \I$; if $b \in \I$ and $c \subseteq b$, then $c \in \I$; if $b \in \I$ and $c \in \I$, then $b \cup c \in \I$; and $a \notin \I$.
It is easily seen that ${\ZZZ}_{0}$ is an ideal on $\omega$.
It is moreover a \emph{P-ideal}, which means that for every collection $\{{a}_{n}: n \in \omega\} \subset {\ZZZ}_{0}$, there exists $a \in {\ZZZ}_{0}$ such that $\forall n \in \omega \[{a}_{n} \; {\subset}^{\ast} \; a\]$, where $X \: {\subset}^{\ast} \: Y$ if and only if $X \setminus Y$ is finite.
${\ZZZ}_{0}$ is also a \emph{tall ideal on $\omega$}, which means that $\forall a \in \cube \exists b \in {\[a\]}^{\omega}\[b \in {\ZZZ}_{0}\]$.
In terms of the Borel hierarchy of $\Pset(\omega)$, ${\ZZZ}_{0}$ is ${F}_{\sigma\delta}$ but not ${G}_{\delta\sigma}$.

Cardinal invariants associated with such tall analytic P-ideals have been studied in several works, principally by Hern{\'a}ndez-Hern{\'a}ndez and Hru{\v{s}}{\'a}k~\cite{michaelstarinv}.
Among the various invariants that have been considered, ${\cov}^{\ast}({\ZZZ}_{0})$ and ${\non}^{\ast}({\ZZZ}_{0})$ are of particular interest.
\begin{Def}\label{def:covandnon}
  \begin{align*}
    &{\cov}^{\ast}({\ZZZ}_{0}) = \min\{\lc \F \rc: \F \subset {\ZZZ}_{0} \wedge \forall a \in \cube \exists b \in \F\[\lc a \cap b \rc = {\aleph}_{0}\]\},\\
    &{\non}^{\ast}({\ZZZ}_{0}) = \min\{\lc \F \rc: \F \subset \cube \wedge \forall b \in {\ZZZ}_{0}\exists a \in \F\[\lc a \cap b \rc < {\aleph}_{0}\]\}.
  \end{align*}
\end{Def}
Of course there is nothing special about ${\ZZZ}_{0}$ here and these invariants can be defined for any tall P-ideal $\I$ on $\omega$.
In fact, these invariants are special cases of the invariants $\cov(\I)$ and $\non(\I)$, which make sense for any ideal $\I$ on any set $X$.
To see how, for each $a \subset \omega$, let $\hat{a} = \{b \subset \omega: \lc b \cap a \rc = {\aleph}_{0}\}$.
This is a ${G}_{\delta}$ subset of $\Pset(\omega)$.
Let ${\hat{\ZZZ}}_{0} = \left\{X \subset \Pset(\omega): \exists a \in {\ZZZ}_{0}\[X \subset \hat{a}\]\right\}$.
Now ${\hat{\ZZZ}}_{0}$ is a $\sigma$-ideal on $\Pset(\omega)$ generated by Borel sets, and it is not hard to show (see Proposition 1.2 of \cite{michaelstarinv}) that $\cov({\hat{\ZZZ}}_{0}) = {\cov}^{\ast}({\ZZZ}_{0})$ and that $\non({\hat{\ZZZ}}_{0}) = {\non}^{\ast}({\ZZZ}_{0})$.

${\ZZZ}_{0}$ turned out to be a critical object of study in \cite{michaelstarinv}, where the invariants associated to ${\ZZZ}_{0}$ were shown to be closely connected to many others, including $\add(\NNN), \cov(\NNN)$, and $\non(\NNN)$.
In that paper, Hern{\'a}ndez-Hern{\'a}ndez and Hru{\v{s}}{\'a}k asked whether ${\cov}^{\ast}({\ZZZ}_{0}) \leq \d$ (Question 3.23(a) of \cite{michaelstarinv}).
Their question was positively answered in \cite{cov_z0}.
Furthermore the proof in \cite{\cov_z0} also yielded the dual inequality $\b \leq {\non}^{\ast}({\ZZZ}_{0})$.
We improve both of these bounds in this paper.
We show that $\min\{\d, \rr\} \leq {\non}^{\ast}({\ZZZ}_{0})$ and that ${\cov}^{\ast}({\ZZZ}_{0}) \leq \max\{\b, \spr\}$, where $\spr$ is a variant of $\s$ that is not known to be distinguishable from $\s$.

The second of our inequalities has implications for what types of forcings can be used to diagonalize $\V \cap {\ZZZ}_{0}$.
Recall that a forcing notion $\P$ in a ground model $\V$ is said to \emph{diagonalize $\V \cap {\ZZZ}_{0}$} if there is an $\mathring{A} \in \VP$ such that ${\forces}_{\P} \mathring{A} \in \cube$ and for each $X \in \V \cap {\ZZZ}_{0}$, ${\forces}_{\P} \lc X \cap \mathring{A} \rc < {\aleph}_{0}$.
Forcings that diagonalize $\V \cap {\ZZZ}_{0}$ tend to increase ${\cov}^{\ast}({\ZZZ}_{0})$.
A celebrated result of Laflamme from \cite{zapping} is that every ${F}_{\sigma}$ ideal can be diagonalized by a proper $\BS$-bounding forcing.
Until the work in \cite{cov_z0}, it was unclear whether a similar result could also be proved for all ${F}_{\sigma\delta}$ P-ideals.
The proof of the inequality ${\cov}^{\ast}({\ZZZ}_{0}) \leq \d$ from \cite{cov_z0} shows that any proper forcing that diagonalizes $\V \cap {\ZZZ}_{0}$ necessarily adds an unbounded real, and since ${\ZZZ}_{0}$ is an ${F}_{\sigma\delta}$ P-ideal, it shows that Laflamme's theorem is, in a certain sense, best possible.
The proof of the inequality ${\cov}^{\ast}({\ZZZ}_{0}) \leq \max\{\b, \spr\}$ given in Section \ref{sec:bounds} has a similar consequence.
It shows that any proper forcing that diagonalizes $\V \cap {\ZZZ}_{0}$ must either add a real that dominates $\V \cap \BS$ or it must add a real that is not promptly split by $\V \cap {(\Pset(\omega))}^{\omega}$ (this notion is introduced in Definition \ref{def:prompt}).
We will also show in Section \ref{sec:variants} that a Suslin c.c.c.\@ forcing cannot add a real that is not promptly split by $\V \cap {(\Pset(\omega))}^{\omega}$, yielding the conclusion that any Suslin c.c.c.\@ poset that diagonalizes $\V \cap {\ZZZ}_{0}$ necessarily adds a dominating real.

The two main inequalities of this paper are obtained by analyzing certain combinatorial variants of the notion of a splitting family.
The first section of this paper is devoted to introducing and studying these variants.
At present, it is unclear if these variants ultimately lead to a new cardinal invariant that is distinguishable from $\s$ (see Question \ref{q:sprsb}).

We end this introduction by fixing some notation that will occur throughout the paper.
$A \subset B$ means $\forall a\[a \in A \implies a \in B\]$.
Thus the symbol ``$\subset$'' does not denote proper subset.
The expression ``$\existsinf x \ldots$'' abbreviates the quantifier ``there exist infinitely many $x$ such that \ldots'', and the dual expression ``$\forallbutfin x \ldots$'' means ``for all but finitely many $x$ \ldots''.
Given a function $f$ and a set $X \subset \dom(f)$, $f''X$ denotes the image of $X$ under $f$ -- that is, $f''X = \{f(x): x \in X\}$.
We use standard cardinal invariants such as $\s$, $\uu$, $\p$, $\rr$, and $\b$, whose definitions may be found in \cite{blasssmall}.
\section{Some variants of the splitting number} \label{sec:variants} 
Several variations on the notion of a splitting family are studied in this section.
One of these variants involves the existence of a type of strong coloring.
It turns out that all of these variations ultimately lead to the same cardinal invariant, which we denote $\spr$.
It will be shown that $\spr$ behaves very similarly to $\s$.
We adopt the convention that for a set $x \subset \omega$, ${x}^{0} = x$ and ${x}^{1} = \omega \setminus x$; this will make certain definitions easier to state.
\begin{Def} \label{def:prompt}
 Let $X = \seq{x}{i}{\in}{\omega}$ be a sequence of elements of $\Pset(\omega)$.
 We say that $X$ \emph{promptly splits} $a$ if for each $n \in \omega$ and each $\sigma \in {2}^{n + 1}$, $\left( {\bigcap}_{i < n + 1}{{x}^{\sigma(i)}_{i}} \right) \cap a$ is infinite.
 A family $\F \subset {\left( \Pset(\omega) \right)}^{\omega}$ is said to be a \emph{promptly splitting family} if for each $a \in {\[\omega\]}^{\omega}$, there exists $X \in \F$ which promptly splits $a$.
\end{Def}
\begin{Def} \label{def:omegas}
 Let $P = \seq{x}{i}{\in}{\omega}$ be a partition of $\omega$ (that is, ${\bigcup}_{i \in \omega}{{x}_{i}} = \omega$ and for any $i < j < \omega$, ${x}_{i} \cap {x}_{j} = 0$).
 We say that $P$ \emph{splits} $a$ if for each $i \in \omega$ ${x}_{i} \cap a$ is infinite.
 A family of partitions $\F$ is called a \emph{splitting family of partitions} if for each $a \in \cube$, there exists $P \in \F$ which splits $a$.
 \begin{align*}
  \spr = \min\{\lc \F \rc: \F \ \text{is a splitting family of partitions}\}.
 \end{align*}
\end{Def}
\begin{Lemma} \label{lem:promptomegas}
 $\spr = \min\{\lc \F \rc: \F \subset {\left( \Pset(\omega) \right)}^{\omega} \wedge \F \ \text{is a promptly splitting family}\}$.
\end{Lemma}
\begin{proof}
 First let $\F \subset {\left( \Pset(\omega) \right)}^{\omega}$ be any promptly splitting family.
 Let $\{ {X}_{\alpha}: \alpha < \kappa \}$ be an enumeration of $\F$, where $\kappa = \lc  \F \rc$.
 For each $\alpha < \kappa$, write ${X}_{\alpha} =  \langle {x}_{\alpha, i}: i < \omega \rangle$, and define ${y}_{\alpha, i} = {x}_{\alpha, i} \setminus i$.
 For each $n \in \omega$, define ${\sigma}_{n} \in {2}^{n + 1}$ as follows: for $i < n$, ${\sigma}_{n}(i) = 0$ and ${\sigma}_{n}(n) = 1$.
 Define ${z}_{\alpha, n} = {\bigcap}_{i < n + 1}{{y}^{{\sigma}_{n}(i)}_{\alpha, i}}$.
 It is easy to see that if $m < n < \omega$, then ${z}_{\alpha, m} \cap {z}_{\alpha, n} = 0$.
 Also for any $l \in \omega$ there is a minimal $n \in \omega$ such that $l \notin {y}_{\alpha, n}$ because ${\bigcap}_{n \in \omega}{{y}_{\alpha, n}} = 0$.
 Then $l \in {z}_{\alpha, n}$, for this minimal $n$.
 Thus ${P}_{\alpha} = \langle {z}_{\alpha, n}: n \in \omega\rangle$ is a partition of $\omega$.
 Moreover it is clear that for any $a \in \cube$, if ${X}_{\alpha}$ promptly splits $a$, then ${P}_{\alpha}$ splits $a$.
 Therefore $\{{P}_{\alpha}: \alpha < \kappa\}$ is a splitting family of partitions.
 
 In the other direction, suppose that $\F$ is any splitting family of partitions.
 Let $\{{P}_{\alpha}: \alpha < \kappa\}$ enumerate $\F$, where $\kappa = \lc \F \rc$, and write ${P}_{\alpha} = \langle {y}_{\alpha, n}: n < \omega \rangle$, for each $\alpha < \kappa$.
 Fix an independent family $\seq{C}{i}{\in}{\omega}$ of subsets of $\omega$.
 For each $\alpha < \kappa$ and $i \in \omega$, define ${x}_{\alpha, i} = {\bigcup}_{n \in {C}_{i}}{{y}_{\alpha, n}}$.
 Note that $\omega \setminus {x}_{\alpha, i} = {\bigcup}_{n \in \omega \setminus {C}_{i}}{{y}_{\alpha, n}}$.
 Put ${X}_{\alpha} = \langle {x}_{\alpha, i}: i < \omega \rangle \in {\left( \Pset(\omega) \right)}^{\omega}$.
 We check that for any $\alpha < \kappa$ and any $a \in \cube$, if ${P}_{\alpha}$ splits $a$, then ${X}_{\alpha}$ promptly splits $a$.
 This would show that $\{{X}_{\alpha}: \alpha < \kappa\}$ is a promptly splitting family and conclude the proof.
 Fix $\alpha < \kappa$ and $a \in \cube$.
 Suppose ${P}_{\alpha}$ splits $a$.
 Fix any $n \in \omega$ and $\sigma \in {2}^{n + 1}$.
 Since $\seq{C}{i}{\in}{\omega}$ is an independent family, ${\bigcap}_{i < n + 1}{{C}^{\sigma(i)}_{i}}$ is non-empty.
 If $m \in {\bigcap}_{i < n + 1}{{C}^{\sigma(i)}_{i}}$, then ${y}_{\alpha, m} \subset {\bigcap}_{i < n + 1}{{x}^{\sigma(i)}_{\alpha, i}}$.
 Since ${y}_{\alpha, m} \cap a$ is infinite, $\left( {\bigcap}_{i < n + 1}{{x}^{\sigma(i)}_{\alpha, i}} \right) \cap a$ is also infinite, as needed.
\end{proof}
Thus the ``$\mathfrak{p}\mathfrak{r}$'' of $\spr$ can either stand for ``partition'' or for ``prompt''.
We next show that $\spr$ is also the least cardinal for which a certain type of strong coloring exists.
\begin{Def} \label{def:tortuous}
 Let $\kappa$ be any cardinal.
 We say that a coloring $c: \kappa \times \omega \times \omega \rightarrow 2$ is \emph{tortuous} if for each $A \in \cube$ and each partition of $\kappa$, $\seq{K}{n}{\in}{\omega}$, there exists $n \in \omega$ such that
 \begin{align} \label{def:tortuouscond}
  \forall \sigma \in {2}^{n + 1} \exists \alpha \in {K}_{n} \exists k \in A\[k > n \wedge \forall i < n + 1\[\sigma(i) = c(\alpha, k, i)\]\].
 \end{align}
 We will say that such a $c$ is a \emph{tortuous coloring on $\kappa$}.
\end{Def}
It is not obvious from the definition that there are tortuous colorings.
The next lemma shows that a tortuous coloring always exists on some cardinal $\leq {2}^{{\aleph}_{0}}$.
\begin{Lemma} \label{lem:tortuous}
 Let $\seq{X}{\alpha}{<}{\kappa}$ be a promptly splitting family.
 There exists a tortuous coloring on $\kappa$.
\end{Lemma}
\begin{proof}
For each $\alpha < \kappa$, write ${X}_{\alpha} = \langle {x}_{\alpha, i}: i < \omega \rangle$.
 Define $c: \kappa \times \omega \times \omega \rightarrow 2$ as follows.
 For any $\alpha < \kappa$, $k, i \in \omega$,
 \begin{align*}
  c(\alpha, k, i) = \begin{cases}
                     0 &\ \text{if} \ k \in {x}_{\alpha, i} \\
                     1 &\ \text{if} \ k \notin {x}_{\alpha, i}.
                    \end{cases}
 \end{align*}
 We check that $c$ is a tortuous coloring.
 Let $A \in \cube$ and suppose $\seq{K}{n}{\in}{\omega}$ is a partition of $\kappa$.
 Suppose $\alpha < \kappa$ is such that ${X}_{\alpha}$ promptly splits $A$.
 Let $n \in \omega$ be such that $\alpha \in {K}_{n}$.
 We check that this $n$ has the required properties.
 Fix $\sigma \in {2}^{n + 1}$.
 As ${X}_{\alpha}$ promptly splits $A$, $\left( {\bigcap}_{i < n + 1}{{x}^{\sigma(i)}_{\alpha, i}} \right) \cap A$ is infinite.
 Choose $k \in \left( {\bigcap}_{i < n + 1}{{x}^{\sigma(i)}_{\alpha, i}} \right) \cap A$ with $k > n$.
 Now for any $i < n + 1$, $c(\alpha, k, i) = 0$ iff $k \in {x}_{\alpha, i}$ iff $\sigma(i) = 0$.
 This concludes the proof.
\end{proof}
By Lemma \ref{lem:tortuous}, there exists a $\kappa$ on which a tortuous coloring exists and the least such $\kappa$ is bounded above by $\spr$.
We next show that the least such $\kappa$ equals $\spr$.
First we show that the definition of a tortuous coloring implies the following self strengthening.
This is the strengthening we will use to prove the bound on ${\cov}^{\ast}({\ZZZ}_{0})$.
\begin{Lemma} \label{lem:strongtortuous}
 Let $\kappa$ be any cardinal and suppose that $c: \kappa \times \omega \times \omega \rightarrow 2$ is a tortuous coloring.
 Then for any $A \in \cube$ there exists $\alpha \in \kappa$ such that for each $n \in \omega$ and $\sigma \in {2}^{n + 1}$, $\existsinf k \in A \forall i < n + 1\[\sigma(i) = c(\alpha, k, i)\]$.
\end{Lemma}
\begin{proof}
 First fix a 1-1 and onto enumeration, $\langle \langle {\sigma}_{k}, {m}_{k} \rangle: k \in \omega \rangle$, of the set ${2}^{< \omega} \times \omega$ such that for each $k \in \omega$, $\lc {\sigma}_{k} \rc \leq k$ and ${m}_{k} \leq k$.
 Now argue by contradiction as follows.
 Let $A \in \cube$ be given and suppose that for each $\alpha \in \kappa$, there exist ${n}_{\alpha} \in \omega$ and ${\sigma}_{\alpha} \in {2}^{{n}_{\alpha} + 1}$ such that 
 \begin{align*}
 \exists {k}_{\alpha} \in \omega \forall k \in A\[k \geq {k}_{\alpha} \implies \exists i < {n}_{\alpha} + 1 \[{\sigma}_{\alpha}(i) \neq c(\alpha, k, i)\]\].
 \end{align*}
 Let ${K}_{n} = \{\alpha \in \kappa: {\sigma}_{\alpha} = {\sigma}_{n} \wedge {k}_{\alpha} = {m}_{n}\}$, for each $n \in \omega$.
 Then $\seq{K}{n}{\in}{\omega}$ is a partition of $\kappa$.
 Applying the definition of a tortuous coloring to $A$ and $\seq{K}{n}{\in}{\omega}$, find $n \in \omega$ satisfying Condition (1) of Definition \ref{def:tortuous}.
 Note that ${\sigma}_{n} \in {2}^{< \omega}$ and that $\lc {\sigma}_{n} \rc \leq n$.
 So there exists $\sigma \in {2}^{n + 1}$ such that ${\sigma}_{n} \subset \sigma$.
 Now we can find $\alpha \in {K}_{n}$ and $k \in A$ such that $k > n$ and $\forall i < n + 1\[\sigma(i) = c(\alpha, k, i)\]$.
 Note that ${\sigma}_{\alpha} = {\sigma}_{n}$, ${k}_{\alpha} = {m}_{n}$, and that $\lc {\sigma}_{\alpha} \rc = {n}_{\alpha} + 1 = \lc {\sigma}_{n} \rc$.
 So ${n}_{\alpha} + 1 \leq n$ and for each $i < {n}_{\alpha} + 1$, ${\sigma}_{\alpha}(i) = \sigma(i)$.
 Moreover, ${k}_{\alpha} = {m}_{n} \leq n < k$.
 Thus for each $i < {n}_{\alpha} + 1$, ${\sigma}_{\alpha}(i) = \sigma(i) = c(\alpha, k, i)$.
 As $k \in A$ and $k > {k}_{\alpha}$, this contradicts the choice of ${k}_{\alpha}$.
 This contradiction concludes the proof.
\end{proof}
\begin{Lemma} \label{lem:prompttortuous}
 $\spr = \min\{\kappa: \text{there is a tortuous coloring on} \ \kappa\}$.
\end{Lemma}
\begin{proof}
Let $\kappa$ be the minimal cardinal on which a tortuous coloring exists.
By Lemmas \ref{lem:promptomegas} and \ref{lem:tortuous}, $\kappa$ exists and is $\leq \spr$.
Let $c: \kappa \times \omega \times \omega \rightarrow 2$ be a tortuous coloring.
We will show that $\spr \leq \kappa$ by producing a promptly splitting family of size at most $\kappa$.
For each $\alpha < \kappa$ and $i < \omega$, define ${x}_{\alpha, i} = \{k \in \omega: c(\alpha, k, i) = 0\}$, and define ${X}_{\alpha} = \langle {x}_{\alpha, i}: i < \omega \rangle \in {\left( \Pset(\omega) \right)}^{\omega}$.
We claim that $\{{X}_{\alpha}: \alpha < \kappa\}$ is promptly splitting.
We will apply Lemma \ref{lem:strongtortuous}.
Fix $A \in \cube$.
Use Lemma \ref{lem:strongtortuous} to find $\alpha \in \kappa$ such that for each $n \in \omega$ and $\sigma \in {2}^{n + 1}$, $\existsinf k \in A \forall i < n + 1\[\sigma(i) = c(\alpha, k, i)\]$.
We claim ${X}_{\alpha}$ promptly splits $A$.
Indeed suppose $n \in \omega$ and $\sigma \in {2}^{n + 1}$.
Then for infinitely many $k \in A$, $\forall i < n + 1\[\sigma(i) = c(\alpha, k, i)\]$.
It is easy to see that each of these infinitely many $k \in A$ belong to $\left( {\bigcap}_{i < n + 1}{{x}^{\sigma(i)}_{\alpha, i}} \right) \cap A$, whence $\left( {\bigcap}_{i < n + 1}{{x}^{\sigma(i)}_{\alpha, i}} \right) \cap A$ is infinite. 
\end{proof}
Next, we show that a very mild guessing principle implies that $\s = \spr$.
The following definition introduces a parametrized version of the combinatorial principle usually denoted \hspace{0.5mm} \begin{rotate}{90} $\mathrel{\hspace*{-2pt}\multimapdot}$ \end{rotate} \hspace{-5pt} (read as ``\emph{stick}'').
This principle was introduced by Broverman et al.\@~\cite{stickdefinition}.
It is known to be strictly weaker than both $\clubsuit$ and $\CH$, but it is also easy to produce models where \hspace{0.5mm} \begin{rotate}{90} $\mathrel{\hspace*{-2pt}\multimapdot}$ \end{rotate} \hspace{-5pt} fails, the model obtained by adding ${\aleph}_{2}$-Cohen reals being an example (see \cite{stickdefinition} for details). 
\begin{Def} \label{def:stick}
  Let $\kappa$, $\lambda$, and $\theta$ be cardinals.
  Then \hspace{0.5mm} \begin{rotate}{90} $\mathrel{\hspace*{-2pt}\multimapdot}$ \end{rotate} \hspace{-5pt} $(\kappa, \lambda, \theta)$ is the following principle: there is a family $\C \subset {\[\kappa\]}^{{\aleph}_{0}}$ of size $\lambda$ such that for any $X \in {\[\kappa\]}^{\theta}$, there exists $A \in \C$ such that $A \subset X$.
\end{Def}
Note that \hspace{0.5mm} \begin{rotate}{90} $\mathrel{\hspace*{-2pt}\multimapdot}$ \end{rotate} \hspace{-6pt} $({\aleph}_{1}, {\aleph}_{1}, {\aleph}_{1})$ is the same as \hspace{0.5mm} \begin{rotate}{90} $\mathrel{\hspace*{-2pt}\multimapdot}$ \end{rotate}.
Several minimal instances of \hspace{0.5mm} \begin{rotate}{90} $\mathrel{\hspace*{-2pt}\multimapdot}$ \end{rotate} \hspace{-5pt} $({\aleph}_{1}, \lambda, {\aleph}_{1})$, \hspace{0.5mm} \begin{rotate}{90} $\mathrel{\hspace*{-2pt}\multimapdot}$ \end{rotate} \hspace{-6pt} $(\kappa, \kappa, {\aleph}_{1})$, and \hspace{0.5mm} \begin{rotate}{90} $\mathrel{\hspace*{-2pt}\multimapdot}$ \end{rotate} \hspace{-7pt} $(\kappa, \kappa, \kappa)$ were studied by Fuchino et al.\@ in \cite{sticksandclubs}.
\begin{Lemma} \label{lem:sticks}
  If \hspace{0.5mm} \begin{rotate}{90} $\mathrel{\hspace*{-2pt}\multimapdot}$ \end{rotate} \hspace{-7pt} $(\s, \s, \p)$ holds, then $\s = \spr$.
\end{Lemma}
\begin{proof}
Fix a splitting family $\seq{x}{\alpha}{<}{\s}$.
Let $\C \subset {\[\s\]}^{{\aleph}_{0}}$ be a family of size $\s$ with the property that for any $X \in {\[\s\]}^{\p}$, there exists $A \in \C$ such that $A \subset X$.
For each $A \in \C$ it is possible to choose ${B}_{A} \subset A$ whose order-type is $\omega$ because $A$ is an infinite set of ordinals.
Let $\langle {\beta}_{A, i}: i \in \omega \rangle$ be the enumeration of ${B}_{A}$ in increasing order.
Define ${y}_{A, 0} = {x}_{{\beta}_{A, 0}}$.
For each $0 < i < \omega$, define ${y}_{A, i} = {x}^{1}_{{\beta}_{A, 0}} \cap \dotsb \cap {x}^{1}_{{\beta}_{A, i - 1}} \cap {x}_{{\beta}_{A, i}}$.
Note that for each $i \in \omega$, ${y}_{A, i} \subset {x}_{{\beta}_{A, i}}$ and that if $j < i$, then ${y}_{A, i} \cap {x}_{{\beta}_{A, j}} = 0$.
So for any $j < i < \omega$, ${y}_{A, i} \cap {y}_{A, j} = 0$.
Now for $i \in \omega$, if $i \in {\bigcup}_{j \in \omega}{{y}_{A, j}}$, then put ${z}_{A, i} = {y}_{A, i}$, else put ${z}_{A, i} = {y}_{A, i} \cup \{ i \}$.
Thus it is clear that ${Z}_{A} = \langle {z}_{A, i}: i \in \omega \rangle$ is a partition of $\omega$.
Let $\F = \{{Z}_{A}: A \in \C\}$.
Then $\F$ is a family of partitions and $\lc \F \rc \leq \lc \C \rc = \s$.
We claim that it is a splitting family of partitions.
To this end fix $a \in \cube$.
We construct a set $X \in {\[\s\]}^{\p}$ as follows.
We will build sequences $\seq{\gamma}{\delta}{<}{\p}$, $\seq{c}{\delta}{<}{\p}$, and $\seq{b}{\delta}{<}{\p}$ such that the following conditions are satisfied for each $\delta < \p$:
\begin{enumerate}
  \item
  ${\gamma}_{\delta} < \s$ and $\forall \xi < \delta\[{\gamma}_{\xi} < {\gamma}_{\delta}\]$;
  \item
  ${c}_{\delta} \in {\[a\]}^{\omega}$ and $\forall \xi < \delta\[{c}_{\delta} \: {\subset}^{\ast} \: {b}_{\xi}\]$;
  \item
  ${b}_{\delta} = {c}_{\delta} \cap {x}^{1}_{{\gamma}_{\delta}}$ and ${\gamma}_{\delta}$ is the least $\alpha < \s$ such that ${x}_{\alpha}$ splits ${c}_{\delta}$.
\end{enumerate}
Suppose for a moment that such sequences can be constructed.
By (1) each ${\gamma}_{\delta} \in \s$ and ${\gamma}_{\xi} \neq {\gamma}_{\delta}$, whenever $\xi \neq \delta$.
Therefore $X = \{{\gamma}_{\delta}: \delta < \p\} \in {\[\s\]}^{\p}$.
Let $A \in \C$ be such that $A \subset X$.
We check that ${Z}_{A}$ splits $a$.
First ${\beta}_{A, 0} = {\gamma}_{\delta}$, for some $\delta < \p$.
Since ${x}_{{\gamma}_{\delta}}$ splits ${c}_{\delta}$ by clause (3), so ${x}^{0}_{{\gamma}_{\delta}} \cap a = {x}^{0}_{{\beta}_{A, 0}} \cap a = {x}_{{\beta}_{A, 0}} \cap a = {y}_{A, 0} \cap a$ is infinite.
Next fix $0 < i < \omega$.
Suppose $0 \leq j \leq i - 1$.
Then ${\beta}_{A, j} < {\beta}_{A, i}$, and so there are $\xi < \delta < \p$ with ${\beta}_{A, j} = {\gamma}_{\xi}$ and ${\beta}_{A, i} = {\gamma}_{\delta}$.
By clauses (2) and (3), ${c}_{\delta} \: {\subset}^{\ast} \: {b}_{\xi} \subset {x}^{1}_{{\gamma}_{\xi}} = {x}^{1}_{{\beta}_{A, j}}$.
Therefore ${c}_{\delta} \: {\subset}^{\ast} \: a \cap {x}^{1}_{{\beta}_{A, 0}} \cap \dotsb \cap {x}^{1}_{{\beta}_{A, i - 1}}$.
Since ${x}_{{\gamma}_{\delta}}$ splits ${c}_{\delta}$, we have that $a \cap {x}^{1}_{{\beta}_{A, 0}} \cap \dotsb \cap {x}^{1}_{{\beta}_{A, i - 1}} \cap {x}^{0}_{{\gamma}_{\delta}} = a \cap {x}^{1}_{{\beta}_{A, 0}} \cap \dotsb \cap {x}^{1}_{{\beta}_{A, i - 1}} \cap {x}_{{\beta}_{A, i}} = a \cap {y}_{A, i}$ is infinite.
Thus we have shown that $\forall i \in \omega\[\lc a \cap {y}_{A, i} \rc = {\aleph}_{0}\]$, which implies that $\forall i \in \omega\[\lc a \cap {z}_{A, i} \rc = {\aleph}_{0}\]$ because $\forall i \in \omega\[{y}_{A, i} \subset {z}_{A, i}\]$.
Hence ${Z}_{A}$ splits $a$, as claimed.

To complete the proof, we show how to construct the sequences satisfying (1)--(3) by induction on $\delta < \p$.
Fix $\delta < \p$ and assume that $\seq{\gamma}{\xi}{<}{\delta}$, $\seq{c}{\xi}{<}{\delta}$, and $\seq{b}{\xi}{<}{\delta}$ satisfying (1)--(3) are given.
Consider any $\zeta < \xi < \delta$.
By clauses (2) and (3), ${b}_{\xi} \subset {c}_{\xi} \: {\subset}^{\ast} \: {b}_{\zeta}$.
So the sequence $\seq{b}{\xi}{<}{\delta}$ is ${\subset}^{\ast}$--descending.
Also for each $\xi < \delta$, ${b}_{\xi} \in {\[a\]}^{\omega}$ because ${x}_{{\gamma}_{\xi}}$ splits ${c}_{\xi}$ and ${b}_{\xi} \subset {c}_{\xi} \subset a$.
Since $\delta < \p$ we can find ${c}_{\delta} \in {\[a\]}^{\omega}$ so that $\forall \xi < \delta\[{c}_{\delta} \: {\subset}^{\ast} \: {b}_{\xi} \]$.
Thus clause (2) is satisfied.
Let ${\gamma}_{\delta}$ be the least $\alpha < \s$ such that ${x}_{\alpha}$ splits ${c}_{\delta}$ and define ${b}_{\delta} =  {x}^{1}_{{\gamma}_{\delta}} \cap {c}_{\delta}$.
Then ${\gamma}_{\delta} < \s$ and clause (3) holds by definition.
So it only remains to check that $\forall \xi < \delta\[{\gamma}_{\xi} < {\gamma}_{\delta}\]$.
Fix $\xi < \delta$ and assume for a contradiction that ${\gamma}_{\delta} \leq {\gamma}_{\xi}$.
Note that ${x}_{{\gamma}_{\delta}}$ splits ${c}_{\xi}$ because ${c}_{\delta} \: {\subset}^{\ast} \: {c}_{\xi}$.
It follows that ${\gamma}_{\xi} \leq {\gamma}_{\delta}$, whence ${\gamma}_{\xi} = {\gamma}_{\delta}$.
However it now follows that ${x}_{{\gamma}_{\xi}}$ splits ${x}^{1}_{{\gamma}_{\xi}}$ because ${c}_{\delta} \: {\subset}^{\ast} \: {b}_{\xi} \subset {x}^{1}_{{\gamma}_{\xi}}$, which is absurd.
This contradiction completes the inductive construction.

Thus $\F$ is a splitting family of partitions.
Since $\lc \F \rc \leq \s$, $\spr \leq \s$, and since $\s \leq \spr$ trivially holds, we conclude that $\s = \spr$.
\end{proof}
We will conclude this section by establishing yet another point of similarity between $\s$ and $\spr$.
We will show that a Suslin c.c.c.\@ forcing cannot increase $\spr$.
This should be compared with the well-known result of Judah and Shelah~\cite{suslinccc} that a Suslin c.c.c.\@ forcing cannot increase $\s$ (see also \cite{BJ}).
Recall the following definitions.
\begin{Def} \label{def:suslinccc}
  A forcing notion $\langle \P, {\leq}_{\P}, {\mathbbm{1}}_{\P}, {\perp}_{\P} \rangle$ is \emph{Suslin c.c.c.\@} if it has the countable chain condition and there exist analytic sets ${R}_{0} \subset \BS$, and ${R}_{1}, {R}_{2} \subset \BS \times \BS$ such that
  \begin{enumerate}
    \item
    $\P = {R}_{0}$;
    \item
    ${\leq}_{\P} = \{\pr{q}{p} \in \P \times \P: q \: {\leq}_{\P} \: p\} = {R}_{1}$;
    \item
    ${\perp}_{\P} = \{\pr{p}{q} \in \P \times \P: \neg\exists r \in \P \[r \: {\leq}_{\P} \: p \wedge r \: {\leq}_{\P} \: q\]\} = {R}_{2}$.
  \end{enumerate}
\end{Def}
Analytic sets are represented as projections of trees.
For any set $A$, if $T \subset {A}^{< \omega}$ is a tree, then $\[T\]$ denotes the set of all branches through $T$, that is $\[T\] = \{f \in {A}^{\omega}: \forall n \in \omega\[f \restrict n \in T\]\}$.
Following standard convention, given a tree $T \subset {(\omega \times \omega)}^{< \omega}$ and $\sigma, \tau \in {\omega}^{n}$ for some $n \in \omega$, we will abuse notation and write $\pr{\sigma}{\tau} \in T$ when what we mean is $\langle \pr{\sigma(i)}{\tau(i)} : i < n \rangle \in T$.
In a related abuse of notation, we write $\pr{f}{g} \in \[T\]$ for some $f, g \in \BS$ when what we mean is $\langle \pr{f(i)}{g(i)}: i \in \omega \rangle \in \[T\]$.
Similar notational conventions apply to subtrees of ${\left( \omega \times \omega \times \omega \right)}^{< \omega}$.
The reader may consult Kechris~\cite{kechrisbook} for further details about representing analytic sets as projections of trees.

For the remainder of this section, fix a Suslin c.c.c.\@ poset $\langle \P, {\leq}_{\P}, {\mathbbm{1}}_{\P}, {\perp}_{\P} \rangle$.
Fix also trees ${T}_{0} \subset {(\omega \times \omega)}^{< \omega}$ and ${T}_{1}, {T}_{2} \subset {\left( \omega \times \omega \times \omega \right)}^{< \omega}$ such that 
\begin{align*}
&\forall p\[p \in \P \iff \exists g \in \BS\[\pr{p}{g} \in \[{T}_{0}\]\]\],\\ 
&\forall p \forall q\[q \: {\leq}_{\P} \: p \iff \exists g \in \BS\[\langle q, p, g \rangle \in \[{T}_{1}\]\]\]\text{, and}\\ &\forall p \forall q\[p \: {\perp}_{\P} \: q \iff \exists g \in \BS \[\langle p, q, g \rangle \in \[{T}_{2}\]\]\].
\end{align*}
\begin{Def} \label{def:Nname}
  Let $\mathring{A}$ be a $\P$-name.
  Suppose that ${\forces}_{\P}{\mathring{A} \in \cube}$.
  Choose a sequence $\AAA = \langle {p}_{m, n}: \pr{m}{n} \in \omega \times \omega \rangle$ and a function $\F: \omega \times \omega \rightarrow 2$ such that:
  \begin{enumerate}
    \item
    for each $n \in \omega$, $\{{p}_{m, n}: m \in \omega\} \subset \P$ is a maximal antichain in $\P$;
    \item
    for each $n, m \in \omega$, ${p}_{m, n} {\forces}_{\P} \: n \in \mathring{A}$ if and only if $\F(m, n) = 1$, while ${p}_{m, n} {\forces}_{\P} \: n \notin \mathring{A}$ if and only if $\F(m, n) = 0$. 
  \end{enumerate}
Define $\mathring{N}(\mathring{A}, \AAA, \F) = \left\{ \pr{\check{n}}{{p}_{m, n}}: n, m \in \omega \wedge \F(m, n) = 1 \right\}$.
Note that $\mathring{N}(\mathring{A}, \AAA, \F)$ is a $\P$-name and that ${\forces}_{\P}\:{\mathring{A} = \mathring{N}(\mathring{A}, \AAA, \F)}$.

Suppose $\mathbf{W}$ is a forcing extension of the universe $\V$.
Then ${\P}^{\mathbf{W}}$, ${\leq}^{\mathbf{W}}_{\P}$, and ${\perp}^{\mathbf{W}}_{\P}$ will denote the reinterpretations in $\mathbf{W}$ of $\P$, ${\leq}_{\P}$, and ${\perp}_{\P}$ respectively.
\end{Def}
It is well-known that $\left\langle {\P}^{\mathbf{W}}, {\leq}^{\mathbf{W}}_{\P}, {\mathbbm{1}}_{\P} , {\perp}^{\mathbf{W}}_{\P} \right\rangle$ is a c.c.c.\@ forcing notion in $\mathbf{W}$ with $\P \subset {\P}^{\mathbf{W}}$, and also that for each $n \in \omega$, $\{{p}_{m ,n}: m \in \omega\} \subset {\P}^{\mathbf{W}}$ is a maximal antichain in ${\P}^{\mathbf{W}}$.
The reader may consult either \cite{suslinccc} or \cite{BJ} for further details.
Note that $\mathring{N}(\mathring{A}, \AAA, \F)$ is a ${\P}^{\mathbf{W}}$-name and that if $H$ is $(\mathbf{W}, {\P}^{\mathbf{W}})$-generic, then $\mathring{N}(\mathring{A}, \AAA, \F)[H] =\{n: n \in \omega \wedge \exists m \in \omega\[\F(m, n) = 1 \wedge {p}_{m, n} \in H \]\}$.
Thus ${\forces}_{{\P}^{\mathbf{W}}}\:{\mathring{N}(\mathring{A}, \AAA, \F) \subset \omega}$ holds in $\mathbf{W}$.
\begin{Lemma} \label{lem:Ninfinite}
  In $\mathbf{W}$, ${\forces}_{{\P}^{\mathbf{W}}}\:{\mathring{N}(\mathring{A}, \AAA, \F) \in \cube}$.
\end{Lemma}
\begin{proof}
  Write $\mathring{N}$ for $\mathring{N}(\mathring{A}, \AAA, \F)$.
  We have remarked above that $\mathring{N}$ is a ${\P}^{\mathbf{W}}$-name and that ${\forces}_{{\P}^{\mathbf{W}}}\:{\mathring{N} \subset \omega}$ holds in $\mathbf{W}$.
  Now in $\mathbf{V}$, we have that for each $p \in \P$ and $l \in \omega$, there exist $n, m \in \omega$ so that $n > l$, $\F(m, n) = 1$, and $p \: {\not\perp}_{\P} \: {p}_{m, n}$, which can be rephrased as
  \begin{align*}
    &\forall p, g \in \BS \forall \langle {g}_{m, n}: \pr{m}{n} \in \omega \times \omega \rangle \in {\left( \BS \right)}^{\omega \times \omega}\\
    &\forall l \in \omega \exists n, m \in \omega\[\pr{p}{g} \in \[{T}_{0}\] \implies \left( n > l \wedge \F(m, n) = 1 \wedge \langle p, {p}_{m, n}, {g}_{m, n} \rangle \notin \[{T}_{2}\] \right) \].
  \end{align*}
  This statement is ${\mathbf{\Pi}}^{1}_{1}$, and so it holds in $\mathbf{W}$.
  Now in $\mathbf{W}$, suppose that $p \in {\P}^{\mathbf{W}}$ and that $l \in \omega$.
  Then we can find $n, m \in \omega$ and $q \in {\P}^{\mathbf{W}}$ so that $n > l$, $\F(m, n) = 1$, and $q \: {\leq}^{\mathbf{W}}_{\P} \: p, {p}_{m, n}$.
  Hence $\pr{\check{n}}{{p}_{m, n}} \in \mathring{N}$ and so $q \: {\forces}_{{\P}^{\mathbf{W}}} \: {n \in \mathring{N}}$.
  Thus we have shown that $\forall p \in {\P}^{\mathbf{W}} \forall l \in \omega \exists n > l \exists q {\leq}^{\mathbf{W}}_{\P} p\[q \: {\forces}_{{\P}^{\mathbf{W}}} \: {n \in \mathring{N}}\]$, which implies that ${\forces}_{{\P}^{\mathbf{W}}}\:{\mathring{N} \ \text{is infinite.}}$
\end{proof}
\begin{Lemma} \label{lem:Nunsplit}
  Suppose $p \in \P$ and that $p \: {\forces}_{\P} \: {\mathring{A} \ \text{is not promptly split by} \ \V \cap {\left( \Pset(\omega) \right)}^{\omega}}$ holds in $\V$.
  Then in $\mathbf{W}$, $p \: {\forces}_{{\P}^{\mathbf{W}}} \: {\mathring{N}(\mathring{A}, \AAA, \F) \ \text{is not promptly split by} \ \mathbf{W} \cap {\left( \Pset(\omega) \right)}^{\omega}}$.
\end{Lemma}
\begin{proof}
Write $\mathring{N}$ for $\mathring{N}(\mathring{A}, \AAA, \F)$.
In $\V$, we have that for each $\bar{p} \: {\leq}_{\P} \: p$ and for each $\seq{x}{i}{\in}{\omega} \in {\left( \Pset(\omega) \right)}^{\omega}$, there exist $q \: {\leq}_{\P} \: \bar{p}$, $k \in \omega$, $\sigma \in {2}^{k + 1}$, $l \in \omega$ such that for each $n, m \in \omega$, if $n \geq l$, $\F(m, n) = 1$, and $n \in {\bigcap}_{i < k + 1}{{x}^{\sigma(i)}_{i}}$, then $q \: {\perp}_{\P} \: {p}_{m, n}$.
This can be rephrased as
\begin{equation*}
 \begin{split}
  &\forall \bar{p}, \bar{g} \in \BS \forall \seq{x}{i}{\in}{\omega} \in {\left( \Pset(\omega) \right)}^{\omega} \exists q, g \in \BS \exists \langle {g}_{m, n}: \pr{m}{n} \in \omega \times \omega \rangle \in {\left( \BS \right)}^{\omega \times \omega}\\
  &\exists k \in \omega \exists \sigma \in {2}^{k + 1} \exists l \in \omega \forall n, m \in \omega\[\langle \bar{p}, p, \bar{g} \rangle \in \[{T}_{1}\] \implies \left(\vphantom{n \in {\bigcap}_{i < k + 1}{{x}^{\sigma(i)}_{i}}}\langle q, \bar{p}, g \rangle \in \[{T}_{1}\] \right.\right. \\
  &\left.\left. \wedge \; \left( \left( n \geq l \wedge \F(m, n) = 1 \wedge n \in {\bigcap}_{i < k + 1}{{x}^{\sigma(i)}_{i}} \right) \implies \langle q, {p}_{m, n}, {g}_{m, n} \rangle \in \[{T}_{2}\] \right) \right)\].
 \end{split} 
\end{equation*}
This is ${\mathbf{\Pi}}^{1}_{2}$.
So by Shoenfield's absoluteness, it continues to holds in $\mathbf{W}$.
Now working in $\mathbf{W}$, fix any $\bar{p} \: {\leq}^{\mathbf{W}}_{\P} \: p$ and $\seq{x}{i}{\in}{\omega} \in {\left( \Pset(\omega) \right)}^{\omega}$.
We know that there are $q \: {\leq}^{\mathbf{W}}_{\P} \: \bar{p}$, $k \in \omega$, $\sigma \in {2}^{k + 1}$, and $l \in \omega$ with the property that for all $n, m \in \omega$, if $n \geq l$, $\F(m, n) = 1$, and $n \in {\bigcap}_{i < k + 1}{{x}^{\sigma(i)}_{i}}$, then $q \: {\perp}^{\mathbf{W}}_{\P} \: {p}_{m, n}$.
We claim that $q \: {\forces}_{{\P}^{\mathbf{W}}} \: {\left( {\bigcap}_{i < k + 1}{{x}^{\sigma(i)}_{i}}  \right) \cap \mathring{N} \subset l}$.
Suppose not.
Then let $H$ be a $(\mathbf{W}, {\P}^{\mathbf{W}})$-generic filter such that $q \in H$ and in $\mathbf{W}\[H\]$, there exists an $n \in \left( {\bigcap}_{i < k + 1}{{x}^{\sigma(i)}_{i}}  \right) \cap \mathring{N}\[H\]$ with $n \notin l$.
By the definition of $\mathring{N}$, $n \in \omega$ and there exists $m \in \omega$ such that $\F(m, n) = 1$ and ${p}_{m, n} \in H$.
However $q \: {\not\perp}^{\mathbf{W}}_{\P} \: {p}_{m, n}$ because they both belong to $H$, which contradicts the choice of $q$ and the fact that $n \geq l$.
This contradiction proves that $q \: {\forces}_{{\P}^{\mathbf{W}}} \: {\left( {\bigcap}_{i < k + 1}{{x}^{\sigma(i)}_{i}}  \right) \cap \mathring{N} \subset l}$ holds in $\mathbf{W}$, which proves that $p \: {\forces}_{{\P}^{\mathbf{W}}} \: {\mathring{N} \ \text{is not promptly split by} \ \mathbf{W} \cap {\left( \Pset(\omega) \right)}^{\omega}}$.
\end{proof}
Recall that if $\langle \RR, {\leq}_{\RR}, {\mathbbm{1}}_{\RR} \rangle$ and $\langle \Ss, {\leq}_{\Ss}, {\mathbbm{1}}_{\Ss} \rangle$ are posets and if $\pi: \RR \rightarrow \Ss$ is a complete embedding, then for any $(\V, \Ss)$-generic filter $H$, ${\pi}^{-1}(H)$ is $(\V, \RR)$-generic.
We can recursively define a map from the $\RR$-names to the $\Ss$-names using $\pi$.
Abusing notation, this map shall also be denoted by $\pi$.
For an $\RR$-name $\mathring{a}$, $\pi(\mathring{a}) = \left\{ \pr{\pi(\mathring{x})}{\pi(p)}: \pr{\mathring{x}}{p} \in \mathring{a} \right\}$.
If $H$ is a $(\V, \Ss)$-generic filter, then for any $\RR$-name, $\mathring{a}$, $\mathring{a}\[{\pi}^{-1}(H)\] = \pi(\mathring{a})\[H\]$, and if $x \in \V$, then $\pi(\check{x}) = \check{x}$, where of course the first ``$\check{x}$'' is with respect to $\langle \RR, {\leq}_{\RR}, {\mathbbm{1}}_{\RR} \rangle$ and the second $\check{x}$ is with respect to $\langle \Ss, {\leq}_{\Ss}, {\mathbbm{1}}_{\Ss} \rangle$.

In the specific case when $\langle \RR, {\leq}_{\RR}, {\mathbbm{1}}_{\RR} \rangle = \langle \Ss, {\leq}_{\Ss}, {\mathbbm{1}}_{\Ss} \rangle$ and $\pi$ is an automorphism,
if $H$ is any $(\V, \RR)$-generic filter, then $\V\[{\pi}^{-1}(H)\] = \V\[H\]$, and moreover for any formula $\varphi({x}_{1}, \dotsc, {x}_{n})$, any ${\mathring{a}}_{1}, \dotsc, {\mathring{a}}_{n} \in {\V}^{\RR}$, and any $r \in \RR$, $r \: {\forces}_{\RR} \: \varphi({\mathring{a}}_{1}, \dotsc, {\mathring{a}}_{n})$ if and only if $\pi(r) \: {\forces}_{\RR} \: \varphi(\pi({\mathring{a}}_{1}), \dotsc, \pi({\mathring{a}}_{n}))$.
\begin{Lemma} \label{lem:auto}
  Let $\langle \RR, {\leq}_{\RR}, {\mathbbm{1}}_{\RR} \rangle$ be a poset that preserves ${\omega}_{1}$.
  Assume that there exist sequences $\seq{\mathring{x}}{i}{<}{\omega}$, $\langle {\pi}_{r, k}: r \in \RR \wedge k \in \omega \rangle$, and $\langle {\pi}_{r, k, \alpha}: r \in \RR \wedge k \in \omega \wedge \alpha \in {\omega}_{1} \rangle$ satisfying the following properties:
  \begin{enumerate}
    \item
    for each $i < \omega$, ${\mathring{x}}_{i}$ is an $\RR$-name such that ${\forces}_{\RR}\:{{\mathring{x}}_{i} \in \cube}$;
    \item
    for each $r \in \RR$, $k \in \omega$, and $\alpha \in {\omega}_{1}$, ${\pi}_{r, k, \alpha}: \RR \rightarrow \RR$ is an automorphism such that ${\pi}_{r, k, \alpha}(r) = r$ and $\forall i < k\[{\forces}_{\RR}\:{{\pi}_{r, k, \alpha}({\mathring{x}}_{i}) = {\mathring{x}}_{i}}\]$;
    \item
    for each $r \in \RR$ and $k \in \omega$, ${\pi}_{r, k}: \RR \rightarrow \RR$ is an automorphism such that ${\pi}_{r, k}(r) = r$, $\forall i < k\[{\forces}_{\RR}\:{{\pi}_{r, k}({\mathring{x}}_{i}) = {\mathring{x}}_{i}}\]$, and ${\forces}_{\RR}\:{\omega \setminus {\pi}_{r, k}({\mathring{x}}_{k}) \: {\subset}^{\ast} \: {\mathring{x}}_{k}}$;
    \item
    for each $r \in \RR$, $k \in \omega$, and $\alpha, \beta \in {\omega}_{1}$, if $\alpha \neq \beta$, then 
    \begin{align*}
     {\forces}_{\RR}\:{\lc {\pi}_{r, k, \alpha}({\mathring{x}}_{k})  \cap {\pi}_{r, k, \beta}({\mathring{x}}_{k})\rc < \omega}.
    \end{align*}
  \end{enumerate}
  Then there is no $p \in \P$ such that $p \: {\forces}_{\P} \: {\mathring{A} \ \text{is not promptly split by} \ \V \cap {\left( \Pset(\omega) \right)}^{\omega}}$.
\end{Lemma}
\begin{proof}
  Assume not.
  Fix $p \in \P$ so that $p \: {\forces}_{\P} \: {\mathring{A} \ \text{is not promptly split by} \ \V \cap {\left( \Pset(\omega) \right)}^{\omega}}$.
  As before write $\mathring{N}$ for $\mathring{N}(\mathring{A}, \AAA, \F)$.
  For the moment, fix a $(\V, \RR)$-generic filter $G$ and let $\mathbf{W} = \VG$.
  Work inside $\mathbf{W}$.
  By Lemma \ref{lem:Nunsplit} and by (1), we know that $p \: {\forces}_{{\P}^{\mathbf{W}}} \: {\langle {\mathring{x}}_{i}\[G\]: i < \omega \rangle \ \text{does not promptly split} \ \mathring{N}}$.
  Let $k \in \omega$ be minimal with the property that there exist $\sigma \in {2}^{k + 1}$ and $q \: {\leq}^{\mathbf{W}}_{\P} \: p$ such that 
  \begin{align*}
   q \: {\forces}_{{\P}^{\mathbf{W}}} \: {\left( {\bigcap}_{i < k + 1}{{\left( {\mathring{x}}_{i}\[G\] \right)}^{\sigma(i)}} \right) \cap \mathring{N} \ \text{is finite.}}
  \end{align*}
  Choose a $\sigma \in {2}^{k + 1}$ witnessing this property of $k$.
  Then 
  \begin{align*}
   p \: {\forces}_{{\P}^{\mathbf{W}}} \: {\left( {\bigcap}_{i < k}{{\left( {\mathring{x}}_{i}\[G\] \right)}^{\sigma(i)}} \right) \cap \mathring{N} \ \text{is infinite}}, 
  \end{align*}
  where ${\bigcap}_{i < k}{{\left( {\mathring{x}}_{i}\[G\] \right)}^{\sigma(i)}}$ is taken to be $\omega$ when $k = 0$, because of the minimality of $k$ and because ${\forces}_{{\P}^{\mathbf{W}}}\:{\mathring{N} \in \cube}$.
  \begin{Claim} \label{claim:auto1}
     $p \: {\forces}_{{\P}^{\mathbf{W}}} \: {\left( {\bigcap}_{i < k}{{\left( {\mathring{x}}_{i}\[G\] \right)}^{\sigma(i)}} \right) \cap \mathring{N} \cap {\left( {\mathring{x}}_{k}\[G\] \right)}^{1} \ \text{is infinite}}$.
  \end{Claim}
  \begin{proof}
   Suppose not.
   Then $q \: {\forces}_{{\P}^{\mathbf{W}}} \: {\left( {\bigcap}_{i < k}{{\left( {\mathring{x}}_{i}\[G\] \right)}^{\sigma(i)}} \right) \cap \mathring{N} \cap {\left( {\mathring{x}}_{k}\[G\] \right)}^{1} \ \text{is finite}}$, for some $q \: {\leq}^{\mathbf{W}}_{\P} \: p$.
   In other words $q \: {\forces}_{{\P}^{\mathbf{W}}} \: {\left( {\bigcap}_{i < k}{{\left( {\mathring{x}}_{i}\[G\] \right)}^{\sigma(i)}} \right) \cap \mathring{N} \: {\subset}^{\ast} \: {\mathring{x}}_{k}\[G\]}$.
   Fix $\mathring{q} \in {\V}^{\RR}$ with $q = \mathring{q}\[G\]$.
   We can find an $r \in G$ such that back in $\V$, $r \: {\forces}_{\RR} \: {\mathring{q} \: {\leq}_{\P}^{\V\[\mathring{G}\]} \: p}$ and
   \begin{align*}
     r \: {\forces}_{\RR} \: {\text{``}\mathring{q} \: {\forces}_{{\P}^{\V\[\mathring{G}\]}} \: {\left( {\bigcap}_{i < k}{{\mathring{x}}^{\sigma(i)}_{i}} \right) \cap \mathring{N} \: {\subset}^{\ast} \: {\mathring{x}}_{k}}\text{''}.}
   \end{align*}
   For each $\alpha \in {\omega}_{1}$, we have that $r \: {\forces}_{\RR} \: {\pi}_{r, k, \alpha}(\mathring{q}) \: {\leq}^{\V\[{\pi}_{r, k, \alpha}(\mathring{G})\]}_{\P} \: p$ and also that
   \begin{align*}
     r \: {\forces}_{\RR} \: {\text{``}{\pi}_{r, k, \alpha}(\mathring{q}) \: {\forces}_{{\P}^{\V\[{\pi}_{r, k, \alpha}(\mathring{G})\]}} \: {\left( {\bigcap}_{i < k}{{\left( {\pi}_{r, k, \alpha}({\mathring{x}}_{i}) \right)}^{\sigma(i)}} \right) \cap \mathring{N} \: {\subset}^{\ast} \: {\pi}_{r, k, \alpha}({\mathring{x}}_{k})}\text{''}.}
   \end{align*}
   Observe that ${\pi}_{r, k, \alpha}(\mathring{G})\[G\] = \mathring{G}\[{\pi}^{-1}_{r, k, \alpha}(G)\] = {\pi}^{-1}_{r, k, \alpha}(G)$, and so $\V\[{\pi}_{r, k, \alpha}(\mathring{G})\[G\]\] = \V\[{\pi}^{-1}_{r, k, \alpha}(G)\] = \V\[G\] = \mathbf{W}$.
   Also by Clause (2), for each $i < k$, ${\pi}_{r, k, \alpha}({\mathring{x}}_{i})\[G\] = {\mathring{x}}_{i}\[G\]$.
   Therefore in $\mathbf{W}$, we have that $\forall \alpha \in {\omega}_{1}\[{\pi}_{r, k, \alpha}(\mathring{q})\[G\] \: {\leq}^{\mathbf{W}}_{\P} \: p\]$ and that for each $\alpha \in {\omega}_{1}$,
   \begin{align*}
    {\pi}_{r, k, \alpha}(\mathring{q})\[G\] \: {\forces}_{{\P}^{\mathbf{W}}} \: {\left( {\bigcap}_{i < k}{{\left( {\mathring{x}}_{i}\[G\] \right)}^{\sigma(i)}} \right) \cap \mathring{N} \: {\subset}^{\ast} \: {\pi}_{r, k, \alpha}({\mathring{x}}_{k})\[G\]}.
   \end{align*}
   Furthermore by Clause (4), for each $\alpha, \beta \in {\omega}_{1}$, if $\alpha \neq \beta$, then 
   \begin{align*}
    \lc {\pi}_{r, k, \alpha}({\mathring{x}}_{k})\[G\] \cap {\pi}_{r, k, \beta}({\mathring{x}}_{k})\[G\] \rc < \omega.
   \end{align*}
   Since $p \: {\forces}_{{\P}^{\mathbf{W}}} \: {\left( {\bigcap}_{i < k}{{\left( {\mathring{x}}_{i}\[G\] \right)}^{\sigma(i)}} \right) \cap \mathring{N} \ \text{is infinite}}$, it follows that $\left\{ {\pi}_{r, k, \alpha}(\mathring{q})\[G\]: \alpha \in {\omega}_{1} \right\}$ is an antichain in ${\P}^{\mathbf{W}}$.
   However this means that ${\P}^{\mathbf{W}}$ is not a c.c.c.\@ poset in $\mathbf{W}$ because $\langle \RR, {\leq}_{\RR}, {\mathbbm{1}}_{\RR} \rangle$ preserves ${\omega}_{1}$ by hypothesis.
   This is a contradiction which proves the claim.
  \end{proof}
  By Claim \ref{claim:auto1}, we can find an $r \in G$ so that in $\V$, 
  \begin{align*}
    r \: {\forces}_{\RR} \: {\text{``} p \: {\forces}_{{\P}^{\V\[\mathring{G}\]}} \: {\left( {\bigcap}_{i < k}{{\mathring{x}}^{\sigma(i)}_{i}} \right) \cap \mathring{N} \cap {\mathring{x}}^{1}_{k} \ \text{is infinite}} \text{''}}.
  \end{align*}
  Applying ${\pi}_{r, k}$, we have that in $\V$
  \begin{align*}
    r \: {\forces}_{\RR} \: {\text{``} p \: {\forces}_{{\P}^{\V\[{\pi}_{r, k}(\mathring{G})\]}} \: {\left( {\bigcap}_{i < k}{{\left( {\pi}_{r, k}({\mathring{x}}_{i}) \right)}^{\sigma(i)}} \right) \cap \mathring{N} \cap {\left( {\pi}_{r, k}({\mathring{x}}_{k}) \right)}^{1} \ \text{is infinite}} \text{''}}. 
  \end{align*}
  Observe that $\V\[{\pi}_{r, k}(\mathring{G})\[G\]\] = \mathbf{W}$ and that for each $i < k$, ${\pi}_{r, k}({\mathring{x}}_{i})\[G\] = {\mathring{x}}_{i}\[G\]$.
  Therefore in $\mathbf{W}$ we have
  \begin{align*}
    p \: {\forces}_{{\P}^{\mathbf{W}}} \: {\left( {\bigcap}_{i < k} {{\left( {\mathring{x}}_{i}\[G\] \right)}^{\sigma(i)}} \right) \cap \mathring{N} \cap {\left( {\pi}_{r, k}({\mathring{x}}_{k})\[G\] \right)}^{1} \ \text{is infinite.}}
  \end{align*}
  By Clause (3), ${\left( {\pi}_{r, k}({\mathring{x}}_{k})\[G\] \right)}^{1} = \omega \setminus {\pi}_{r, k}({\mathring{x}}_{k})\[G\] \: {\subset}^{\ast} \: {\mathring{x}}_{k}\[G\]$.
  Therefore
  \begin{align*}
    p \: {\forces}_{{\P}^{\mathbf{W}}} \: {\left( {\bigcap}_{i < k}{{\left( {\mathring{x}}_{i}\[G\] \right)}^{\sigma(i)}} \right) \cap \mathring{N} \cap {\mathring{x}}_{k}\[G\] \ \text{is infinite.}}
  \end{align*}
  However this together with Claim \ref{claim:auto1} gives a contradiction because by the choice of $k$ and $\sigma$, there exists $q \: {\leq}^{\mathbf{W}}_{\P} \: p$ such that
  \begin{align*}
    q \: {\forces}_{{\P}^{\mathbf{W}}} \: {\left( {\bigcap}_{i < k}{{\left( {\mathring{x}}_{i}\[G\] \right)}^{\sigma(i)}} \right) \cap \mathring{N} \cap {\left( {\mathring{x}}_{k}\[G\] \right)}^{\sigma(k)} \ \text{is finite.}}
  \end{align*}
This contradiction concludes the proof.
\end{proof}
\begin{Theorem} \label{thm:suslinprompt}
  $\langle \P, {\leq}_{\P}, {\mathbbm{1}}_{\P}, {\perp}_{\P} \rangle$ does not add any real that is not promptly split by $\V \cap {\left( \Pset(\omega) \right)}^{\omega}$.
\end{Theorem}
\begin{proof}
  If not, then there would be a $\P$-name $\mathring{A}$ such that ${\forces}_{\P}{\mathring{A} \in \cube}$ and a $p \in \P$ such that $p \: {\forces}_{\P} \: {\mathring{A} \ \text{is not promptly split by} \ \V \cap {\left( \Pset(\omega) \right)}^{\omega}}$.
  In view of Lemma \ref{lem:auto}, in order to get a contradiction, it suffices to find a c.c.c.\@ poset $\langle \RR, {\leq}_{\RR}, {\mathbbm{1}}_{\RR} \rangle$ together with sequences $\seq{\mathring{x}}{i}{<}{\omega}$, $\langle {\pi}_{r, k}: r \in \RR \wedge k \in \omega \rangle$, and $\langle {\pi}_{r, k, \alpha}: r \in \RR \wedge k \in \omega \wedge \alpha \in {\omega}_{1} \rangle$ satisfying Clauses (1)--(4) there.
  Define $\RR$ to be the collection of all $r$ such that $r$ is a function, $\lc r \rc < \omega$, $\dom(r) \subset \omega \times \omega$, $\ran(r) \subset \omega$, and $\forall \pr{l}{i}, \pr{l}{j} \in \dom(r)\[i \neq j \implies r(l, i) \neq r(l, j) \]$.
  Define $s \: {\leq}_{\RR} \: r$ if and only if $s, r \in \RR$ and $s \supset r$, and define ${\mathbbm{1}}_{\RR} = \emptyset$.
  Obviously $\pr{\RR}{{\leq}_{\RR}, {\mathbbm{1}}_{\RR}}$ is a c.c.c.\@ poset.
  Define $E = \{m \in \omega: m \ \text{is even}\}$ and $O = \{m \in \omega: m \ \text{is odd}\}$.
  Also, for each $r \in \RR$, fix ${L}_{r} \in \omega$ with $\ran(r) \subset {L}_{r}$.
  Fix a $(\V, \RR)$-generic filter $G$ for a moment.
  In $\V\[G\]$, $F = \bigcup G$ is a function from $\omega \times \omega$ to $\omega$ with the property that for each $\pr{l}{i}, \pr{l}{j} \in \omega \times \omega$, if $i \neq j$, then $F(l, i) \neq F(l, j)$.
  Therefore for any $l \in \omega$ and any finite $T \subset \omega$, $\{i < \omega: F(l, i) \in T\}$ is finite.
  For each $l \in \omega$, define ${x}_{l} = \{i \in \omega: F(l, i) \in E \}$.
  It is clear that ${x}_{l} \in \cube$ for every $l \in \omega$.
  Unfixing $G$, back in $\V$, let $\mathring{F}$ be an $\RR$-name such that ${\forces}_{\RR}\:{\mathring{F} = \bigcup\mathring{G}}$, and let $\seq{\mathring{x}}{l}{<}{\omega}$ be a sequence of $\RR$-names such that for each $l < \omega$, ${\forces}_{\RR}\:{{\mathring{x}}_{l} = \{i \in \omega: \mathring{F}(l, i) \in E\}}$.
  Then ${\forces}_{\RR}\:{{\mathring{x}}_{l} \in \cube}$, for all $l \in \omega$.
  
  Now suppose that $f: \omega \rightarrow \omega$ is a permutation and that $k \in \omega$.
  We define a function ${\pi}_{f, k}: \RR \rightarrow \RR$ as follows.
  Let $r \in \RR$ be given.
  Then ${\pi}_{f, k}(r)$ is the function such that $\dom({\pi}_{f, k}(r)) = \dom(r)$ and for every $\pr{l}{i} \in \dom({\pi}_{f, k}(r))$, ${\pi}_{f, k}(r)(l, i) = f(r(l, i))$ when $k = l$, while ${\pi}_{f, k}(r)(l, i) = r(l, i)$ when $l \neq k$.
  It is easy to check that ${\pi}_{f, k}$ is an automorphism.
  Furthermore for each $r \in \RR$, if $\forall m \in {L}_{r}\[f(m) = m\]$, then ${\pi}_{f, k}(r) = r$.
  Fix a $(\V, \RR)$-generic filter $G$.
  Then ${\pi}_{f, k}(\mathring{G})\[G\] = \mathring{G}\[{\pi}^{-1}_{f, k}(G)\] = {\pi}^{-1}_{f, k}(G) = \{r \in \RR: {\pi}_{f, k}(r) \in G\} = \{{\pi}_{{f}^{-1}, k}(s): s \in G\}$.
  Therefore ${\pi}_{f, k}(\mathring{F})\[G\] = \bigcup{\pi}_{f, k}(\mathring{G})\[G\] = \bigcup\{{\pi}_{{f}^{-1}, k}(s): s \in G\}$.
  It follows that for any $\pr{l}{i} \in \omega \times \omega$, ${\pi}_{f, k}(\mathring{F})\[G\](l, i) = {f}^{-1}(\mathring{F}\[G\](l, i))$ when $l = k$, while ${\pi}_{f, k}(\mathring{F})\[G\](l, i) = \mathring{F}\[G\](l, i)$ when $l \neq k$.
  So for every $l \in \omega$ with $l \neq k$, ${\pi}_{f, k}({\mathring{x}}_{l})\[G\] = {\mathring{x}}_{l}\[G\]$, and ${\pi}_{f, k}({\mathring{x}}_{k})\[G\] = \{i \in \omega: \mathring{F}\[G\](k, i) \in f''E\}$.
  In particular, unfixing $G$ and going back to $\V$, we have that for each $l \in \omega$, if $l \neq k$, then ${\forces}_{\RR}\:{{\pi}_{f, k}({\mathring{x}}_{l}) = {\mathring{x}}_{l}}$.
  
  Now, working in $\V$, fix an almost disjoint family $\{{A}_{\alpha}: \alpha < {\omega}_{1}\}$ of infinite subsets of $\omega$.
  Let $r \in \RR$ and $k \in \omega$ be fixed.
  Let $f: \omega \rightarrow \omega$ be a permutation such that $\forall m \in {L}_{r}\[f(m) = m\]$, $f''(E \setminus {L}_{r}) = O \setminus {L}_{r}$, and $f''(O \setminus {L}_{r}) = E \setminus {L}_{r}$.
  Define ${\pi}_{r, k} = {\pi}_{f, k}$.
  Also for each $\alpha < {\omega}_{1}$, choose a permutation ${f}_{\alpha}: \omega \rightarrow \omega$ such that $\forall m \in {L}_{r}\[{f}_{\alpha}(m) = m\]$, ${f}_{\alpha}''(E \setminus {L}_{r}) = {A}_{\alpha} \setminus {L}_{r}$, and ${f}_{\alpha}''(O \setminus {L}_{r}) = \omega\setminus({A}_{\alpha} \cup {L}_{r})$. 
  For each $\alpha \in {\omega}_{1}$, define ${\pi}_{r, k, \alpha} = {\pi}_{{f}_{\alpha}, k}$.
  In light of the observations already made, it suffices to check that ${\forces}_{\RR}\:{\omega \setminus {\pi}_{f, k}({\mathring{x}}_{k}) \: {\subset}^{\ast} \: {\mathring{x}}_{k}}$, and that for any $\alpha, \beta \in {\omega}_{1}$, if $\alpha \neq \beta$, then ${\forces}_{\RR}\:{\lc {\pi}_{{f}_{\alpha}, k}({\mathring{x}}_{k}) \cap {\pi}_{{f}_{\beta}, k}({\mathring{x}}_{k}) \rc < \omega}$.
  To this end, consider an arbitrary $(\V, \RR)$-generic filter $G$.
  In $\V\[G\]$, it is clear that $\left( \omega \setminus \left( {\pi}_{f, k}({\mathring{x}}_{k})\[G\] \right) \right) \setminus \left( {\mathring{x}}_{k}\[G\] \right) \subset \{i \in \omega: \mathring{F}\[G\](k, i) \in {L}_{r}\}$, which is a finite set.
  Similarly, if $\alpha, \beta \in {\omega}_{1}$ and $\alpha \neq \beta$, then ${\pi}_{{f}_{\alpha}, k}({\mathring{x}}_{k})\[G\] \cap {\pi}_{{f}_{\beta}, k}({\mathring{x}}_{k})\[G\] \subset \{i \in \omega: \mathring{F}\[G\](k, i) \in {L}_{r} \cup ({A}_{\alpha} \cap {A}_{\beta})\}$.
  By almost disjointness, ${L}_{r} \cup ({A}_{\alpha} \cap {A}_{\beta})$ is a finite subset of $\omega$.
  Therefore $\{i \in \omega: \mathring{F}\[G\](k, i) \in {L}_{r} \cup ({A}_{\alpha} \cap {A}_{\beta})\}$ is finite as well.
  Hence ${\pi}_{{f}_{\alpha}, k}({\mathring{x}}_{k})\[G\] \cap {\pi}_{{f}_{\beta}, k}({\mathring{x}}_{k})\[G\]$ is a finite set.
  This establishes everything that is needed for the proof of the theorem.
\end{proof}
It is well known that every new real that is added by a finite support iteration of Suslin c.c.c.\@ posets is actually added by a countable fragment of the iteration, and this countable fragment itself can be coded as a Suslin c.c.c.\@ poset (see, for example, \cite{suslinccc} for a proof).
Hence we get the following corollary to Theorem \ref{thm:suslinprompt}, which is analogous to a result of Judah and Shelah for the splitting number.
\begin{Cor} \label{cor:suslinspr}
  A finite support iteration of Suslin c.c.c.\@ posets does not increase $\spr$.
\end{Cor}
If $\I$ is any analytic ideal on $\omega$, then the Mathias and Laver forcings associated with $\I$ are examples of Suslin c.c.c.\@ posets.
So, in particular, finite support iterations of Mathias and Laver forcings associated with analytic ideals do not increase $\spr$.
\begin{Question} \label{q:sprsb}
  Is $\s = \spr$?
  Is $\spr \leq \max\{\b, \s\}$?
\end{Question}
\section{A bound for ${\cov}^{\ast}({\ZZZ}_{0})$} \label{sec:bounds}
The two main inequalities of the paper saying that ${\cov}^{\ast}({\ZZZ}_{0}) \leq \max\{\b, \spr\}$ and $\min\{\d, \rr\} \leq {\non}^{\ast}({\ZZZ}_{0})$ will be proved in this section. 
We will need a few lemmas proved in \cite{cov_z0} for our construction.
We state these below without proof and refer the reader to \cite{cov_z0} for details.
\begin{Lemma} [Lemma 12 of \cite{cov_z0}]\label{lem:regular}
 Let $I$ be an interval partition.
 Let $A \subset \omega$ be such that for each $l \geq 0$, there exists $N \in \omega$ such that for each $n \geq N$:
 \begin{enumerate}
  \item
  $\frac{\lc A \cap {I}_{n} \rc}{\lc {I}_{n} \rc} \leq {2}^{-l}$;
  \item
  $\forall i, j \in A \cap {I}_{n}\[i \neq j \implies \lc i - j \rc > {2}^{l - 1}\]$.
 \end{enumerate}
 Then $A$ has density $0$.
\end{Lemma}
\begin{Lemma}[Lemma 13 of \cite{cov_z0}]\label{lem:tree}
 Let $l$ be a member of $\omega$ greater than $0$ and let $X \subset \omega$ with $\lc X \rc = {2}^{l}$.
 Then there exists a sequence $\{{A}_{\sigma}: \sigma \in {2}^{\leq l}\}$ such that:
 \begin{enumerate}
  \item
  $\forall m \leq l\[{\bigcup}_{\sigma \in {2}^{m}}{A}_{\sigma} = X \wedge \forall \sigma, \tau \in {2}^{m}\[\sigma \neq \tau \implies {A}_{\sigma} \cap {A}_{\tau} = 0 \]\]$;
  \item
  $\forall \sigma \in {2}^{\leq l}\[\lc {A}_{\sigma} \rc = {2}^{l - \lc \sigma \rc}\]$ and $\forall \sigma, \tau \in {2}^{\leq l}\[\sigma \subset \tau \implies {A}_{\tau} \subset {A}_{\sigma}\]$;
  \item
  for each $\sigma \in {2}^{\leq l}$, $\forall i, j \in {A}_{\sigma}\[i \neq j \implies \lc i - j \rc > {2}^{\lc \sigma \rc - 1}\]$.
 \end{enumerate}
\end{Lemma}
\begin{Def}[Definition 15 of \cite{cov_z0}]\label{def:FJ}
 Let $J$ be an interval partition such that for each $n \in \omega$ there exists ${l}_{n} \in \omega$ such that ${l}_{n} > 0$, ${l}_{n} \geq n$, and $\lc {J}_{n} \rc = {2}^{{l}_{n}}$.
 Applying Lemma \ref{lem:tree}, fix a sequence $\bar{A} = \langle {A}_{n, \sigma}: n \in \omega \wedge \sigma \in {2}^{\leq {l}_{n}}\rangle$ such that for each $n \in \omega$, the sequence $\{{A}_{n, \sigma}: \sigma \in {2}^{\leq {l}_{n}}\}$ satisfies (1)--(3) of Lemma \ref{lem:tree} with $l$ as ${l}_{n}$ and $X$ as ${J}_{n}$.
 Define ${\F}_{J, \bar{A}}$ to be the collection of all functions $f \in \BS$ such that for each $n \in \omega$ and $l < {l}_{n}$, there exists $\sigma \in {2}^{l + 1}$ such that ${f}^{-1}(\{l\}) \cap {J}_{n} = {A}_{n, \sigma}$, and there exists $\tau \in {2}^{{l}_{n}}$ such that ${f}^{-1}(\{{l}_{n}\}) \cap {J}_{n} = {A}_{n, \tau}$.
\end{Def}
\begin{remark} \label{rem:FJ}
Observe that if $f \in {\F}_{J, \bar{A}}$, then for each $n \in \omega$ and $k \in {J}_{n}$, $f(k) \leq {l}_{n}$.
Also for any $n, l \in \omega$,
\begin{align*}
 \displaystyle\frac{\lc \left\{ k \in {J}_{n}: f(k) \geq l \right\} \rc}{\lc {J}_{n} \rc} \leq {2}^{-l},
\end{align*}
and for any $i, j \in \left\{ k \in {J}_{n}: f(k) \geq l \right\}$, if $i \neq j$, then $\lc i - j \rc > {2}^{l - 1}$.
Moreover for any $f \in {\F}_{J, \bar{A}}$, $n \in \omega$, and $l \leq {l}_{n}$, there is ${\sigma}_{f, n, l} \in {2}^{l}$ such that ${A}_{n, {\sigma}_{f, n, l}} = \{k \in {J}_{n}: f(k) \geq l\}$.
\end{remark}
The next lemma is a simple variation of a standard fact.
However the proof we give below is slightly more cumbersome than the standard proof because of our need to ensure Clause (2), which says that the size of each interval is equal to an exact power of $2$.
\begin{Lemma} \label{lem:b}
 There exists a family $B$ of interval partitions such that:
 \begin{enumerate}
  \item
  $\lc B \rc \leq \b$;
  \item
  for each $I \in B$ and for each $n \in \omega$, there exists ${l}_{n} \in \omega$ such that ${l}_{n} > 0$, ${l}_{n} \geq n$, and $\lc {I}_{n} \rc = {2}^{{l}_{n}}$;
  \item
  for any interval partition $J$, there exists $I \in B$ such that $\existsinf n \in \omega \exists k > n\[{J}_{k} \subset {I}_{n}\]$.
 \end{enumerate}
\end{Lemma}
\begin{proof}
 For each $f \in \BS$ define an interval partition ${I}_{f} = \langle {i}_{f, n}: n \in \omega \rangle$ as follows.
 Define ${i}_{f, 0} = 0$, and given ${i}_{f, n} \in \omega$, let $L = \max\{({i}_{f, n}) + 1, f(n + 1)\}$.
 Find ${l}_{n} \in \omega$ such that ${l}_{n} > 0$, ${l}_{n} \geq n$, and ${2}^{{l}_{n}} \geq L - {i}_{f, n}$.
 Define ${i}_{f, n + 1} = {2}^{{l}_{n}} + {i}_{f, n}$.
 Note that ${i}_{f, n} < ({i}_{f, n}) + 1 \leq L \leq {i}_{f, n + 1}$.
 Note also that $f(n + 1) \leq L \leq {i}_{f, n + 1}$.
 This completes the definition of ${I}_{f}$, which is clearly an interval partition.
 For each $n \in \omega$, $\lc {I}_{f, n} \rc = {i}_{f, n + 1} - {i}_{f, n} = {2}^{{l}_{n}}$, for some ${l}_{n} \in \omega$ with ${l}_{n} > 0$ and ${l}_{n} \geq n$.
 
 Now suppose $U \subset \BS$ is an unbounded family with $\lc U \rc = \b$.
 Put $B = \{{I}_{f}: f \in U\}$.
 Clauses (1) and (2) hold by construction.
 So we verify (3).
 Let $J = \seq{j}{n}{\in}{\omega}$ be any interval partition.
 For each $k \in \omega$, define ${g}_{k} \in \BS$ by ${g}_{k}(n) = {j}_{k + n}$, for all $n \in \omega$.
 Let $g \in \BS$ be such that $\forall k \in \omega\[{g}_{k} \; {\leq}^{\ast} \; g\]$.
 Since $U$ is unbounded, find $f \in U$ such that $X = \{n \in \omega: f(n) > g(n)\}$ is infinite.
 We check that ${I}_{f}$ has the required properties.
 Fix $N \in \omega$.
 Choose $m > N + 1$ such that ${j}_{m} \geq {i}_{f, N + 1}$.
 Let $k = m - N - 1 \geq 1$.
 By choice of $g$, there exists ${N}_{k} \in \omega$ such that $\forall n \geq {N}_{k}\[{j}_{k + n} \leq g(n)\]$.
 Let $M = \max\{N + 1, {N}_{k}\}$.
 Since $X$ is infinite, there exists $n \in X$ with $n \geq M$.
 For any such $n$, ${j}_{k + n} \leq g(n) < f(n) \leq {i}_{f, n}$.
 So we conclude that there exists $n \geq N + 1$ such that ${j}_{k + n} < {i}_{f, n}$.
 Let $n$ be the minimal number with this property.
 Note that $N + 1$ does not have this property because ${j}_{k + N + 1} = {j}_{m} \geq {i}_{f, N + 1}$.
 So $n > N + 1$ and so $n - 1 \geq N + 1$.
 It follows by the minimality of $n$ that ${i}_{f, n - 1} \leq {j}_{k + n - 1} < {j}_{k + n} < {i}_{f, n}$.
 Therefore, ${J}_{k + n - 1} \subset {I}_{f, n - 1}$.
 Note that $k + n - 1 > n - 1$ because $k \geq 1$ and also that $n - 1 > N$.
 Thus we have proved that $\forall N \in \omega \exists l > N \exists l' > l\[{J}_{l'} \subset {I}_{f, l}\]$, which establishes (3).
\end{proof}
\begin{Def} \label{def:Z}
 Let $J$ be any interval partition such that for each $n \in \omega$, there exists ${l}_{n} \in \omega$ such that ${l}_{n} > 0$, ${l}_{n} \geq n$, and $\lc {J}_{n} \rc = {2}^{{l}_{n}}$.
 Let $\bar{A}$ and ${\F}_{J, \bar{A}}$ be as in Definition \ref{def:FJ}.
 For any interval partition $I$, function $f \in {\F}_{J, \bar{A}}$, and $l \in \omega$, define ${Z}_{I, J, f, l} = \{m \in \omega: \exists k \in {I}_{l}\[m \in {J}_{k} \wedge f(m) \geq l\]\}$.
 Define ${Z}_{I, J, f} = {\bigcup}_{l \in \omega}{{Z}_{I, J, f, l}}$.
\end{Def}
\begin{Lemma} \label{lem:Zdensity}
 For any $I, J$, and $f$ as in Definition \ref{def:Z}, ${Z}_{I, J, f}$ has density $0$.
\end{Lemma}
\begin{proof}
 We apply Lemma \ref{lem:regular} with $J$ and ${Z}_{I, J, f}$ as the $I$ and the $A$ of Lemma \ref{lem:regular} respectively.
 To check clauses (1) and (2) of Lemma \ref{lem:regular}, fix $x \geq 0$, a member of $\omega$.
 Let $N = {i}_{x} \in \omega$, and suppose $n \geq N$ is given.
 Then by the definition of ${Z}_{I, J, f}$, ${Z}_{I, J, f} \cap {J}_{n} \subset \{m \in {J}_{n}: f(m) \geq x\}$.
 Hence by Remark \ref{rem:FJ}, $\frac{\lc {Z}_{I, J, f} \cap {J}_{n} \rc}{\lc {J}_{n} \rc} \leq \frac{\lc \{m \in {J}_{n}: f(m) \geq x\}\rc}{\lc {J}_{n} \rc} \leq {2}^{-x}$, as required for clause (1).
 Also, $\forall i, j \in {Z}_{I, J, f} \cap {J}_{n}\[i \neq j \implies \lc i - j \rc > {2}^{x - 1}\]$, as required for clause (2).
 Thus by Lemma \ref{lem:regular}, ${Z}_{I, J, f}$ has density $0$.
\end{proof}
\begin{Lemma} \label{lem:bounded}
Let $J$, $\bar{A}$, and ${\F}_{J, \bar{A}}$ be as in Definition \ref{def:FJ}.
Let $B$ be a family of interval partitions satisfying (1)--(3) of Lemma \ref{lem:b}.
Fix $f \in {\F}_{J, \bar{A}}$.
Suppose $X \subset \omega$ is such that for each $I \in B$, $X \cap {Z}_{I, J, f}$ is finite.
Then there exists $n \in \omega$ such that $f''X \subset n$.
\end{Lemma}
\begin{proof}
 Suppose for a contradiction that for each $n \in \omega$, there exists $m \in X$ such that $f(m) \geq n$.
 Define an interval partition $K = \seq{k}{n}{\in}{\omega}$ as follows.
 ${k}_{0} = 0$ and suppose that ${k}_{n} \in \omega$ is given, for some $n \in \omega$.
 Define $N = \max\left( \left\{f(m): m \in {\bigcup}_{k < {k}_{n}}{{J}_{k}}\right\} \cup \left\{n\right\} \right)$.
 By hypothesis, there exists $m \in X$ such that $f(m) \geq N + 1$.
 Choose such an $m \in X$ and let $k$ be such that $m \in {J}_{k}$.
 Note that ${k}_{n} \leq k$ by the definition of $N$.
 Define ${k}_{n + 1} = k + 1$.
 This completes the definition of $K$.
 Note that $\forall n \in \omega \exists k \in {K}_{n} \exists m \in X \cap {J}_{k}\[f(m) > n\]$.
 By clause (3) of Lemma \ref{lem:b}, there is an interval partition $I \in B$ such that $\existsinf l \in \omega \exists n > l\[{K}_{n} \subset {I}_{l}\]$.
 Consider any $l \in \omega$ for which there exists $n > l$ such that ${K}_{n} \subset {I}_{l}$.
 There exist $k \in {I}_{l}$ and $m \in X \cap {J}_{k}$ such that $f(m) > l$.
 It follows that $m \in X \cap {Z}_{I, J, f, l}$.
 Thus we conclude that $\existsinf l \in \omega\[X \cap {Z}_{I, J, f, l} \neq 0\]$, contradicting the hypothesis that $X \cap {Z}_{I, J, f}$ is finite, for all $I \in B$.
\end{proof}
\begin{Def} \label{def:fc}
 Let $J$ and $\bar{A}$ be as in Definition \ref{def:FJ}.
 Suppose $C: \omega \rightarrow {2}^{< \omega}$ and that for each $n \in \omega$, $\dom(C(n)) \geq {l}_{n}$.
 For each $l < {l}_{n}$, define ${\sigma}_{n, l} = {\left( C(n) \restrict l \right)}^{\frown}{\langle 1 - C(n)(l) \rangle} \in {2}^{l + 1}$, and define ${\sigma}_{n, {l}_{n}} = C(n)\restrict {l}_{n} \in {2}^{{l}_{n}}$.
 Note that for all $l < l' \leq {l}_{n}$, ${A}_{n, {\sigma}_{n, l}} \cap {A}_{n, {\sigma}_{n, l'}} = 0$ and that ${\bigcup}_{l \leq {l}_{n}}{{A}_{n,{\sigma}_{n, l}}} = {J}_{n}$.
 Let ${f}_{C}: \omega \rightarrow \omega$ be defined as follows.
 Given $n \in \omega$ and $k \in {J}_{n}$, ${f}_{C}(k) = l$, where $l$ is the unique number $l \leq {l}_{n}$ such that $k \in {A}_{n, {\sigma}_{n, l}}$.
 It is easy to check that ${f}_{C} \in {\F}_{J, \bar{A}}$
\end{Def}
\begin{Theorem} \label{thm:main1}
 Let $\kappa$ be a cardinal on which a tortuous coloring exists.
 Then ${\cov}^{\ast}({\ZZZ}_{0}) \leq \max\{\kappa, \b\}$.
\end{Theorem}
\begin{proof}
 Let $c: \kappa \times \omega \times \omega \rightarrow 2$ be a tortuous coloring.
 Fix any interval partition $J$ with the property that for each $n \in \omega$, there exists ${l}_{n} \in \omega$ such that ${l}_{n} > 0$, ${l}_{n} \geq n$, and $\lc {J}_{n} \rc = {2}^{{l}_{n}}$.
 Let $\bar{A}$ be as in Definition \ref{def:FJ} (with respect to $J$).
 For each $\alpha \in \kappa$, define ${C}_{\alpha}: \omega \rightarrow {2}^{< \omega}$ as follows.
 Given $n \in \omega$, ${C}_{\alpha}(n)$ is the function in ${2}^{{l}_{n}}$ such that for each $l < {l}_{n}$, ${C}_{\alpha}(n)(l) = c(\alpha, n, l)$.
 Define ${f}_{\alpha} = {f}_{{C}_{\alpha}} \in {\F}_{J, \bar{A}}$.
 Fix a family $B$ of interval partitions satisfying clauses (1)--(3) of Lemma \ref{lem:b}.
 For each $I \in B$ and $\alpha \in \kappa$, let ${Z}_{I, \alpha} = {Z}_{I, J, {f}_{\alpha}}$.
 By Lemma \ref{lem:Zdensity}, each ${Z}_{I, \alpha}$ has density $0$.
 Let $\GGG = \{{Z}_{I, \alpha}: I \in B \wedge \alpha \in \kappa\}$ and note that $\lc \GGG \rc \leq \max\{\kappa, \b\}$.
 We will show that $\forall X \in \cube \exists I \in B \exists \alpha \in \kappa\[\lc X \cap {Z}_{I, \alpha} \rc = \omega\]$.
 Thus $\GGG$ will witness that ${\cov}^{\ast}({\ZZZ}_{0}) \leq \max\{\kappa, \b\}$.
 Fix $X \in \cube$.
 Assume for a contradiction that $X \cap {Z}_{I, \alpha}$ is finite for all $I \in B$ and $\alpha \in \kappa$.
 $L = \{n \in \omega: {J}_{n} \cap X \neq 0 \}$ is infinite because $X$ is infinite.
 For each $n \in L$, there exists ${\tau}_{n} \in {2}^{{l}_{n}}$ such that $X \cap {A}_{n, {\tau}_{n}} \neq 0$. 
 By Lemma \ref{lem:bounded} for each $\alpha \in \kappa$ there exists ${n}_{\alpha} \in \omega$ such that ${f}_{\alpha}'' X \subset {n}_{\alpha}$.
 Next, for each $n \in L$ let ${x}_{n}$ be the member of ${2}^{\omega}$ such that ${x}_{n} \restrict {l}_{n} = {\tau}_{n}$ and $\forall l \in \omega \setminus {l}_{n}\[{x}_{n}(l) = 0\]$.
 Find $A \in {\[L\]}^{\omega}$ and $x \in {2}^{\omega}$ such that $\seq{x}{n}{\in}{A}$ converges to $x$.
 Apply Lemma \ref{lem:strongtortuous} to find $\alpha \in \kappa$ such that for each $n \in \omega$ and $\sigma \in {2}^{n + 1}$, $\existsinf k \in A \forall i < n + 1\[\sigma(i) = c(\alpha, k, i)\]$.
 Let $n = {n}_{\alpha}$ and $\sigma = x \restrict n + 1$.
 By convergence, there exists ${k}^{\ast} \in \omega$ such that $\forall k \in A\[k \geq {k}^{\ast} \implies {x}_{k} \restrict n + 1 = x \restrict n + 1\]$.
 Fix $k \in A$ such that $k \geq {k}^{\ast}$, $k > n$, and $\forall i < n + 1\[\sigma(i) = c(\alpha, k, i)\]$.
 It is easy to see that ${\tau}_{k} \restrict n + 1 = {C}_{\alpha}(k) \restrict n + 1$.
 It follows from the definition of ${f}_{\alpha} = {f}_{{C}_{\alpha}}$ that for each $x \in {A}_{k, {\tau}_{k} \restrict n + 1}$, ${f}_{\alpha}(x) > n$.
 However $X \cap {A}_{k, {\tau}_{k} \restrict n + 1} \neq 0$ because $X \cap {A}_{k, {\tau}_{k}} \neq 0$.
 Therefore there exists $x \in X$ such that ${f}_{\alpha}(x) > n = {n}_{\alpha}$, contradicting the fact that ${f}_{\alpha}''X \subset {n}_{\alpha}$.
 This concludes the proof.
\end{proof}
\begin{Cor} \label{cor:main1}
 ${\cov}^{\ast}({\ZZZ}_{0}) \leq \max\{\spr, \b\}$.
\end{Cor}
Suppose $\V$ is a ground model.
Suppose that the coloring $c$ used in the proof of Theorem \ref{thm:main1} is defined in $\V$ from $\V \cap {(\Pset(\omega))}^{\omega}$ following the procedure of Lemma \ref{lem:tortuous}, and that the family of interval partitions $B$ is defined in $\V$ from $\V \cap \BS$ via the procedure of Lemma \ref{lem:b}.
Let $\VG$ be a forcing extension of $\V$.
If there is a set $X \in \cube$ in $\VG$ such that $Z \cap X$ is finite for all $Z \in \V \cap {\ZZZ}_{0}$, then it follows from the proof of Theorem \ref{thm:main1} that either $\V \cap {(\Pset(\omega))}^{\omega}$ is no longer a promptly splitting family or that $\V \cap \BS$ is no longer an unbounded family in $\VG$.
So we get the following corollary.
\begin{Cor} \label{cor:diagonalizez0}
  Let $\P \in \V$ be a forcing notion that diagonalizes $\V \cap {\ZZZ}_{0}$.
  Then either $\P$ adds an element of $\BS$ that dominates $\V \cap \BS$ or it adds an element of $\cube$ that is not promptly split by $\V \cap {(\Pset(\omega))}^{\omega}$.
\end{Cor}
If $\P$ is a Suslin c.c.c.\@ poset, then the second possibility is ruled out by Theorem \ref{thm:suslinprompt}.
Furthermore if $\P = \langle {\P}_{\alpha}; {\mathring{\Q}}_{\alpha}: \alpha \leq \delta \rangle$ is a finite support iteration of c.c.c.\@ posets and if each iterand preserves all unbounded families, then $\P$ does not increase $\b$.
If $\P$ is also not allowed to increase $\spr$, then of course $\P$ cannot increase ${\cov}^{\ast}({\ZZZ}_{0})$.
\begin{Cor} \label{cor:suslinbcovz_0}
  If a Suslin c.c.c.\@ poset in $\V$ diagonalizes $\V \cap {\ZZZ}_{0}$, then it necessarily adds a dominating real.
  If $\P = \langle {\P}_{\alpha}; {\mathring{\Q}}_{\alpha}: \alpha \leq \delta \rangle$ is a finite support iteration of Suslin c.c.c.\@ posets and if each iterand preserves all unbounded families, then $\P$ does not increase ${\cov}^{\ast}({\ZZZ}_{0})$.
\end{Cor}
An example of a Suslin c.c.c.\@ forcing which preserves all unbounded families is the Mathias forcing associated to an ${F}_{\sigma}$ filter (see Canjar~\cite{canjar1}).
So a consequence of Corollary \ref{cor:suslinbcovz_0} is that finite support iterations of Mathias forcings of ${F}_{\sigma}$ filters do not increase ${\cov}^{\ast}({\ZZZ}_{0})$.

The next result dualizes Corollary \ref{cor:main1}.
However we do not need any variant of $\rr$ because of the following fact, which says that any family of fewer than $\rr$ many members of $\cube$ can be simultaneously promptly split.
\begin{Lemma} \label{lem:rsplit}
 Suppose $\F \subset \cube$ is a family of size less than $\rr$.
 Then there exists a sequence $X = \seq{x}{k}{<}{\omega} \in {\left( \Pset(\omega) \right)}^{\omega}$ such that $X$ promptly splits $A$, for each $A \in \F$.
\end{Lemma}
\begin{proof}
 If $\F$ is empty then any $X \in {\left( \Pset(\omega) \right)}^{\omega}$ vacuously satisfies the conclusion of the lemma.
 So we may assume that $\F$ is non-empty.
 We define a sequence $\seq{y}{i}{\in}{\omega}$ as follows.
 Use the assumption that $\F$ has size less than $\rr$ to find ${y}_{0} \subset \omega$ such that both ${y}_{0} \cap A$ and $\left( \omega \setminus {y}_{0} \right) \cap A$ are infinite, for each $A \in \F$.
 Next suppose that for some $n \in \omega$, a sequence $\seq{y}{i}{\leq}{n} \in {\Pset(\omega)}^{n + 1}$ is given such that both ${y}_{n} \cap A$ and $\left( \omega \setminus \left( {\bigcup}_{i \leq n}{{y}_{i}} \right) \right) \cap A$ are infinite, for each $A \in \F$.
 As $\F$ is non-empty, $\left( \omega \setminus \left( {\bigcup}_{i \leq n}{{y}_{i}} \right) \right)$ is an infinite subset of $\omega$, and $\GGG = \left\{\left( \omega \setminus \left( {\bigcup}_{i \leq n}{{y}_{i}} \right) \right) \cap A: A \in \F\right\}$ is a collection of infinite subsets of $\left( \omega \setminus \left( {\bigcup}_{i \leq n}{{y}_{i}} \right) \right)$ of size less than $\rr$.
 So we can find ${y}_{n + 1} \subset \left( \omega \setminus \left( {\bigcup}_{i \leq n}{{y}_{i}} \right) \right)$ such that both ${y}_{n + 1} \cap B$ and $\left( \left( \omega \setminus \left( {\bigcup}_{i \leq n}{{y}_{i}} \right) \right) \setminus {y}_{n + 1} \right) \cap B$ are infinite, for each $B \in \GGG$.
 It is clear that $\seq{y}{i}{\leq}{n + 1}$ satisfies the inductive hypothesis.
 This concludes the construction of $\seq{y}{i}{\in}{\omega}$.
 Note that $\seq{y}{i}{\in}{\omega}$ is a pairwise disjoint sequence.
 Fix a independent family $\seq{C}{k}{\in}{\omega}$ of subsets of $\omega$.
 For each $k \in \omega$, define ${x}_{k} = {\bigcup}_{i \in {C}_{k}}{{y}_{i}}$.
 This is a subset of $\omega$, and we claim that $\seq{x}{k}{\in}{\omega}$ promptly splits $A$, for each $A \in \F$.
 Indeed, fix $A \in \F$.
 Suppose $n \in \omega$ and $\sigma \in {2}^{n + 1}$.
 Then ${\bigcap}_{k < n + 1}{{C}^{\sigma(k)}_{k}} \neq 0$.
 Let $i \in {\bigcap}_{k < n + 1}{{C}^{\sigma(k)}_{k}}$.
 Since ${y}_{i} \subset {\bigcap}_{k < n + 1}{{x}^{\sigma(k)}_{k}}$ and ${y}_{i} \cap A$ is infinite, $\left( {\bigcap}_{k < n + 1}{{x}^{\sigma(k)}_{k}} \right) \cap A$ is infinite as well, as needed.
\end{proof}
\begin{Lemma} \label{lem:lessd}
 Let $\kappa < \d$ be a cardinal.
 Suppose $\seq{I}{\alpha}{<}{\kappa}$ is a sequence of interval partitions.
 Then there exists an interval partition $J$ such that for each $\alpha < \kappa$, $\existsinf n \in \omega \exists k > n\[{I}_{\alpha, k} \subset {J}_{n}\]$.
\end{Lemma}
\begin{proof}
 This is very similar to Lemma \ref{lem:b}.
 For each $\alpha < \kappa$ and $l \in \omega$, define ${f}_{\alpha, l}(n) = {i}_{\alpha, (n \; + \; l)}$, for each $n \in \omega$.
 Now $\{{f}_{\alpha, l}: \alpha < \kappa \wedge l < \omega\}$ is a family of functions of size less than $\d$.
 So there exists $g \in \BS$ such that for each $\alpha < \kappa$ and $l < \omega$, $\existsinf n \in \omega\[{f}_{\alpha, l}(n) < g(n)\]$.
 Define $J$ as follows.
 Put ${j}_{0} = 0$ and suppose ${j}_{n} \in \omega$ is given for some $n \in \omega$.
 Define ${j}_{n + 1} = \max\{{j}_{n} + 1, g(n + 1)\}$.
 It is clear that $J$ is an interval partition.
 We check that it is as required.
 So fix $\alpha < \kappa$ and $N \in \omega$.
 We will find $n > N$ and $k > n$ such that ${I}_{\alpha, k} \subset {J}_{n}$.
 Fix $m > N + 1$ such that ${i}_{\alpha, m} \geq {j}_{N + 1}$ and let $l = m - N - 1$.
 Note $l \geq 1$.
 By choice of $g$, there exists $M \geq N + 1$ such that $g(M) > {i}_{\alpha, (M \; + \; l)}$.
 Now ${j}_{M} \geq g(M) > {i}_{\alpha, (M \; + \; l)}$ because $M > 0$.
 So we conclude that there exists $M$ with the property that $M \geq N + 1$ and ${j}_{M} > {i}_{\alpha, (M \; + \; l)}$.
 Let $M$ be minimal with this property.
 Note that $N + 1$ does not have this property, so $M > N + 1$.
 Put $n = M - 1$ and $k = n + l$.
 It follows that $n \geq N + 1$ and that ${j}_{n} \leq {i}_{\alpha, k} < {i}_{\alpha, k + 1} < {j}_{n + 1}$, and so ${I}_{\alpha, k} \subset {J}_{n}$.
 Since $n > N$ and $k > n$, we are done.
\end{proof}
\begin{Theorem} \label{thm:main2}
 $\min\{\d, \rr\} \leq {\non}^{\ast}({\ZZZ}_{0})$.
\end{Theorem}
\begin{proof}
 Let $\GGG$ be any family of infinite subsets of $\omega$ with $\lc \GGG \rc < \min\{\d, \rr\}$.
 We aim to produce a $Z \in {\ZZZ}_{0}$ such that $\forall B \in \GGG\[\lc B \cap Z \rc = {\aleph}_{0}\]$.
 Fix any interval partition $J$ such that for each $n \in \omega$, there exists ${l}_{n} \in \omega$ such that ${l}_{n} > 0$, ${l}_{n} \geq n$, and $\lc {J}_{n} \rc = {2}^{{l}_{n}}$.
 Let $\bar{A}$ be as in Definition \ref{def:FJ} with respect to $J$.
 Fix $B \in \GGG$.
 Define ${L}_{B} = \{n \in \omega: {J}_{n} \cap B \neq 0 \}$.
 As $B$ is infinite, ${L}_{B}$ is infinite.
 For each $n \in {L}_{B}$, let ${\tau}^{B}_{n} \in {2}^{{l}_{n}}$ be such that ${A}_{n, {\tau}^{B}_{n}} \cap B \neq 0$.
 For each $n \in {L}_{B}$, define ${x}^{B}_{n}$ to be the element of ${2}^{\omega}$ such that ${x}^{B}_{n} \restrict {l}_{n} = {\tau}^{B}_{n}$ and $\forall l \geq {l}_{n}\[{x}^{B}_{n}(l) = 0\]$.
 Now we can find ${U}_{B} \in {\[{L}_{B}\]}^{\omega}$ and ${x}^{B} \in {2}^{\omega}$ such that $\langle {x}^{B}_{n}: n \in {U}_{B} \rangle$ converges to ${x}^{B}$.
 Unfix $B$ and consider $\F = \{{U}_{B}: B \in \GGG\}$.
 Then $\F \subset \cube$ and $\lc \F \rc < \rr$.
 Therefore by Lemma \ref{lem:rsplit}, there exists a sequence $\seq{z}{k}{<}{\omega} \in {\left( \Pset(\omega) \right)}^{\omega}$ which promptly splits ${U}_{B}$, for each $B \in \GGG$.
 Now define $C: \omega \rightarrow {2}^{< \omega}$ as follows.
 For $n \in \omega$, $C(n)$ is the function from ${l}_{n}$ to $2$ such that for each $k < {l}_{n}$, $C(n)(k) = 0$ iff $n \in {z}_{k}$.
 $C$ satisfies the conditions of Definition \ref{def:fc}.
 Therefore ${f}_{C} \in {\F}_{J, \bar{A}}$, where ${f}_{C}$ is defined in Definition \ref{def:fc}
 
 Fix any $B \in \GGG$ and $l \in \omega$.
 We will produce a $y \in B$ such that ${f}_{C}(y) \geq l$.
 Since $\langle {x}^{B}_{n}: n \in {U}_{B} \rangle$ converges to ${x}^{B}$, there exists $N \in \omega$ such that $\forall n \in {U}_{B}\[n \geq N \implies {x}^{B}_{n} \restrict \left( l + 1 \right) = {x}^{B} \restrict \left( l + 1 \right) \]$.
 Also since $\seq{z}{k}{\in}{\omega}$ promptly splits ${U}_{B}$, $\left( {\bigcap}_{k < l + 1}{{z}^{{x}^{B}(k)}_{k}} \right) \cap {U}_{B}$ is infinite.
 Choose $n \in \left( {\bigcap}_{k < l + 1}{{z}^{{x}^{B}(k)}_{k}} \right) \cap {U}_{B}$ such that $n \geq N$ and $n > l$.
 Note that ${l}_{n} \geq n > l$ and that for each $k < l + 1$, $C(n)(k) = 0$ iff $n \in {z}_{k}$ iff ${x}^{B}(k) = 0$. 
 Thus $C(n) \restrict \left( l + 1 \right) = {x}^{B} \restrict \left( l + 1 \right) = {x}^{B}_{n} \restrict \left( l + 1 \right) = {\tau}^{B}_{n} \restrict \left( l + 1 \right)$.
 For notational convenience, write $\sigma = {\tau}^{B}_{n} \restrict \left( l + 1 \right)$.
 Since ${A}_{n, {\tau}^{B}_{n}} \subset {A}_{n, \sigma}$ and since ${A}_{n, {\tau}^{B}_{n}} \cap B \neq 0$, we can choose a $y \in B \cap {A}_{n, \sigma}$.
 We claim that ${f}_{C}(y) \geq l$.
 By the definition of ${f}_{C}$, it suffices to prove that for each $l' < l$, $y \notin {A}_{n, {\sigma}_{n, l'}}$, where ${\sigma}_{n, l'}$ is defined in Definition \ref{def:fc} (with respect to $C$).
 To see this, fix any $l' < l$.
 Put $\eta = \sigma \restrict \left( l' + 1 \right)$.
 Then $\eta(l') = C(n)(l') \neq 1 - C(n)(l') = {\sigma}_{n, l'}(l')$.
 Thus ${\sigma}_{n, l'}, \eta \in {2}^{l' + 1}$ and $\eta \neq {\sigma}_{n, l'}$.
 Therefore ${A}_{n, {\sigma}_{n, l'}} \cap {A}_{n, \eta} = 0$.
 On the other hand ${A}_{n, \sigma} \subset {A}_{n, \eta}$ because $\eta \subset \sigma$.
 Hence $y \in {A}_{n, \eta}$, whence $y \notin {A}_{n, {\sigma}_{n, l'}}$ as claimed.
 
 The argument of the previous paragraph shows that ${f}_{C}$ is unbounded on every $B \in \GGG$.
 Now for each $B \in \GGG$ define an interval partition ${I}_{B}$ as follows.
 Let ${i}_{B, 0} = 0$ and suppose that for some $n \in \omega$, ${i}_{B, n} \in \omega$ is given.
 Define $M = \max\left( \{{f}_{C}(y) + 1: y \in {\bigcup}_{m \leq {i}_{B, n}}{{J}_{m}} \} \cup \{n\} \right)$.
 Let $y \in B$ be such that ${f}_{C}(y) \geq M$.
 Let $m \in \omega$ be such that $y \in {J}_{m}$.
 Note that $m > {i}_{B, n}$.
 Define ${i}_{B, n + 1} = m + 1$.
 This concludes the definition of ${I}_{B}$.
 Note that for each $n \in \omega$, $\exists m \in {I}_{B, n} \exists y \in {J}_{m} \cap B\[{f}_{C}(y) \geq n\]$.
 Now $\{{I}_{B}: B \in \GGG\}$ is a family of interval partitions of size less than $\d$.
 Therefore by Lemma \ref{lem:lessd}, there is an interval partition $I$ such that for each $B \in \GGG$, $\existsinf k \in \omega \exists n > k\[{I}_{B, n} \subset {I}_{k}\]$.
 Let $Z = {Z}_{I, J, {f}_{C}}$.
 Then $Z$ has density $0$ because ${f}_{C} \in {\F}_{J, \bar{A}}$.
 To complete the proof of the theorem, we show that $\lc Z \cap B \rc = {\aleph}_{0}$, for every $B \in \GGG$.
 To this end, fix any $B \in \GGG$.
 Then $Y = \{k \in \omega: \exists n > k\[{I}_{B, n} \subset {I}_{k}\]\}$ is infinite by choice of $I$.
 Consider any $k \in Y$ and let $n > k$ be such that ${I}_{B, n} \subset {I}_{k}$.
 There exist $m \in {I}_{B, n}$ and $y \in {J}_{m} \cap B$ with ${f}_{C}(y) \geq n$.
 By definition of ${Z}_{I, J, {f}_{C}, k}$, $y \in B \cap {Z}_{I, J, {f}_{C}, k}$.
 Thus $B \cap {Z}_{I, J, {f}_{C}, k} \neq 0$, for every $k \in Y$.
 When $k < k' < \omega$, then ${Z}_{I, J, {f}_{C}, k} \cap {Z}_{I, J, {f}_{C}, k'} = 0$.
 It follows that $B \cap Z$ is infinite, as claimed.
\end{proof}
We point out here that it is provable in $\ZFC$ that $\min\{\d, \rr\} = \min\{\d, \uu\}$.
We do not know if this observation was already known, however a closely related observation was made by Mildenberger who showed that $\rr \geq \min\{\uu, \mathfrak{g}\}$.
More details about Mildenberger's work may be found on Page 452 of \cite{blasssmall}.
\begin{Lemma} \label{lem:dur}
 $\min\{\d, \rr\} = \min\{\d, \uu\}$.
\end{Lemma}
\begin{proof}
 It is well-known (see \cite{blasssmall}) that $\rr \leq \uu$.
 Therefore $\min\{\d, \rr\} \leq \min\{\d, \uu\}$.
 We will prove that $\min\{\d, \uu\} \leq \min\{\d, \rr\}$.
 Let $\kappa = \min\{\d, \rr\}$ and assume for a contradiction that $\kappa < \min\{\d, \uu\}$.
 We argue that $\kappa < \rr$.
 Let $\{{X}_{\xi}: \xi < \kappa\}$ be any family of elements of $\cube$.
 Clearly $\lc {\[\kappa\]}^{< \omega} \rc = \kappa$.
 So let $\seq{u}{\alpha}{<}{\kappa}$ enumerate ${\[\kappa\]}^{< \omega}$.
 For each $\alpha < \kappa$ choose an interval partition ${I}_{\alpha}$ such that $\forall n \in \omega \forall \xi \in {u}_{\alpha}\[{X}_{\xi} \cap {I}_{\alpha, n} \neq 0\]$.
 Since $\kappa < \min\{\d, \uu\} \leq \d$, Lemma \ref{lem:lessd} applies and implies that there is an interval partition $J$ such that for each $\alpha < \kappa$, $\existsinf n \in \omega \exists k > n\[{I}_{\alpha, k} \subset {J}_{n}\]$.
 Now for each $\alpha < \kappa$ define ${A}_{\alpha}$ to be $\{n \in \omega: \forall \xi \in {u}_{\alpha}\[{X}_{\xi} \cap {J}_{n} \neq 0\]\}$.
 Each ${A}_{\alpha}$ is infinite by the choice of $J$.
 Define $\F = \{B \in \Pset(\omega): \exists \alpha < \kappa\[{A}_{\alpha} \; {\subset}^{\ast} \; B\]\}$.
 We check that $\F$ is a non-principal filter on $\omega$.
 First each element of $\F$ is infinite.
 Next, if $B \in \F$ and $B \subset C \subset \omega$, then $C \in \F$.
 Finally suppose that $B, C \in \F$.
 Let $\alpha, \beta < \kappa$ be such that ${A}_{\alpha} \; {\subset}^{\ast} \; B$ and ${A}_{\beta} \; {\subset}^{\ast} \; C$.
 Let $\gamma < \kappa$ be such that ${u}_{\gamma} = {u}_{\alpha} \cup {u}_{\beta}$.
 Then ${A}_{\gamma} \; {\subset}^{\ast} \; B \cap C$, showing that $B \cap C \in \F$.
 This checks that $\F$ is a non-principal filter on $\omega$.
 Since $\kappa < \min\{\d, \uu\} \leq \uu$ and $\F$ is generated by at most $\kappa$ many elements, it cannot be an ultrafilter.
 So fix $B \subset \omega$ such that neither $B$ nor $\omega \setminus B$ belongs to $\F$.
 Let $Y = {\bigcup}_{n \in B}{{J}_{n}}$.
 Fix any $\xi < \kappa$ and suppose $\alpha < \kappa$ is such that ${u}_{\alpha} = \{\xi\}$.
 Since ${A}_{\alpha}$ is neither almost included in $B$ nor in $\omega \setminus B$, it follows that $\lc {A}_{\alpha} \cap \left( \omega \setminus B \right) \rc = \lc {A}_{\alpha} \cap B \rc = \omega$.
 Therefore $\existsinf n \in B\[{X}_{\xi} \cap {J}_{n} \neq 0 \]$ and $\existsinf n \in \left( \omega \setminus B \right)\[{X}_{\xi} \cap {J}_{n} \neq 0\]$.
 It follows that $\lc {X}_{\xi} \cap Y \rc= \lc {X}_{\xi} \cap \left( \omega \setminus Y \right) \rc = \omega$.
 Thus $Y \subset \omega$ and it reaps the family $\{{X}_{\xi}: \xi < \kappa\}$.
 We have proved that any family of at most $\kappa$ many elements of $\cube$ can be reaped, whence $\rr > \kappa = \min\{\d, \rr\}$.
 However this together with the hypothesis that $\kappa < \min\{\d, \uu\}$ implies that $\d = \kappa < \min\{\d, \uu\} \leq \d$, which is a contradiction.
\end{proof}
We thus get the following ``improvement'' of Theorem \ref{thm:main2}.
\begin{Cor} \label{cor:dunonZ_0}
$\min\{\d, \uu\} \leq {\non}^{\ast}({\ZZZ}_{0})$.
\end{Cor}
\section{Questions} \label{sec:q}
An outstanding open question is about the connection between ${\cov}^{\ast}({\ZZZ}_{0})$ and $\b$, which is closely related to what forcings can diagonalize $\V \cap {\ZZZ}_{0}$.
\begin{Question} \label{q:cov_z0b}
  Is ${\cov}^{\ast}({\ZZZ}_{0}) \leq \b$?
  Is $\d \leq {\non}^{\ast}({\ZZZ}_{0})$?
  Is there a proper forcing which diagonalizes $\V \cap {\ZZZ}_{0}$ while preserving all unbounded families?
\end{Question}
\def\polhk#1{\setbox0=\hbox{#1}{\ooalign{\hidewidth
  \lower1.5ex\hbox{`}\hidewidth\crcr\unhbox0}}}
\providecommand{\bysame}{\leavevmode\hbox to3em{\hrulefill}\thinspace}
\providecommand{\MR}{\relax\ifhmode\unskip\space\fi MR }
\providecommand{\MRhref}[2]{%
  \href{http://www.ams.org/mathscinet-getitem?mr=#1}{#2}
}
\providecommand{\href}[2]{#2}

%\bibliographystyle{amsplain}
%\bibliography{Bibliography} 
\end{document}